\documentclass[leqno,11pt]{amsart}

\usepackage{palatino}
\usepackage[mathcal]{euler}

\usepackage{amssymb,amsmath,amsthm,enumerate,slashed}
\usepackage[latin1]{inputenc}
\usepackage[T1]{fontenc}
\usepackage{mathrsfs}

\usepackage{xcolor}

\definecolor{DarkBlue}{rgb}{0.2,0.2,0.4}

\usepackage{xy}	
\xyoption{all}

\usepackage{marginnote}
\usepackage[mmddyyyy]{datetime}

\usepackage[colorlinks=true, linkcolor=DarkBlue, urlcolor= BrickRed, citecolor=DarkBlue]{hyperref}

\newcommand{\Z}{\mathbb{Z}}
\newcommand{\R}{\mathbb{R}}
\newcommand{\C}{\mathbb{C}}

\newcommand{\Cs}{\ensuremath{\mathrm{C}^*}}
\newcommand{\Csr}{\ensuremath{\mathrm{C}_r^*}}
\newcommand{\Gr}{\widehat{G}_{\mathrm{temp}}}
\newcommand{\spin}{\text{\rm spin}}
\newcommand{\fin}{\text{\rm fin}}

\DeclareMathOperator{\Ad}{Ad}
\DeclareMathOperator{\ad}{ad}
\DeclareMathOperator{\Ind}{Ind}

\DeclareMathOperator{\Cliff}{Cliff}

\DeclareMathOperator{\infch}{inf.ch.}

\DeclareMathOperator{\Index}{Index}
\DeclareMathOperator{\module}{mod}

    \newcommand{\Compact}{\mathfrak{K}}

\renewcommand{\Re}{\mathop{\rm{Re}}}
\renewcommand{\Im}{\mathop{\rm{Im}}}

\newcommand{\Bullet}{\,\,\,\,\,\bullet\,\,\,\,\,}

\swapnumbers
\numberwithin{equation}{section}

\theoremstyle{plain}
\newtheorem*{theorem*}{Theorem}
\newtheorem*{lemma*}{Lemma}

\newtheorem{theorem}[equation]{Theorem}
\newtheorem{ck-conjecture}[equation]{Connes-Kasparov Isomorphism} 
\newtheorem{lemma}[equation]{Lemma}
\newtheorem{proposition}[equation]{Proposition}
\newtheorem{corollary}[equation]{Corollary}
\theoremstyle{definition}
\newtheorem{definition}[equation]{Definition}
\newtheorem{example}[equation]{Example}

\newtheorem{remark}[equation]{Remark}
\newtheorem{remarks}[equation]{Remarks}
\newtheorem*{remark*}{Remark}

\title[On the Connes-Kasparov Isomorphism, I]{On the Connes-Kasparov Isomorphism, I:\\ The Reduced C*-algebra of a Real Reductive Group and the K-theory of the Tempered Dual}
\author{Pierre Clare}
\author{Nigel Higson}
\author{Yanli Song}
\author{Xiang Tang}

\begin{document}

\date{\today}

\begin{abstract}
This is the first of two papers dedicated to the detailed determination of the reduced $C^*$-algebra of a connected, linear, real reductive  group up to Morita equivalence, and a new and very explicit proof of the Connes-Kasparov conjecture for these groups using representation theory. In this part we shall  give details of the $C^*$-algebraic Morita equivalence and then explain how the Connes-Kasparov morphism in operator $K$-theory may be computed using what we call the matching theorem, which is a purely representation-theoretic result. We shall prove our matching theorem in the sequel, and indeed go further by giving  a simple, direct construction of the components of the tempered dual that have nontrivial $K$-theory using  David Vogan's approach to the classification of the tempered dual.
\end{abstract}

\maketitle


\section{Introduction}

If $\pi$ is a unitary representation of a locally compact group $G$ on a Hilbert space $H$, then the formula
\begin{equation}
\label{eq-rep-of-the-group-algebra}
\pi (f) = \int _G f(g)\pi (g) \; dg \qquad (f\in C_c^\infty (G) )
\end{equation}
defines a  representation of the 
group $C^*$-algebra $C^*(G)$ as bounded operators on $H$. In this way the category of unitary representations of $G$ becomes equivalent to the category of (nondegenerate) representations of $C^*(G)$. See   \cite[Ch.~13]{DixmierEnglish}. 
 
The $C^*$-algebra point of view equips the unitary dual of $G$  with a topology whose closed sets are in bijection with   two-sided ideals  $J\triangleleft C^*(G)$: the closed set determined by $J$ is the set of all irreducible unitary representations that vanish on $J$.  The \emph{reduced dual} of $G$ is by definition the closed subset of the unitary dual that is associated to the kernel of the left regular representation
\begin{equation}
\label{eq-regular-rep}
\lambda \colon C^*(G) \longrightarrow  \mathfrak{B} (L^2 (G)).
\end{equation}
If $G$ is a real reductive group, then the representations in the reduced dual are precisely Harish-Chandra's \emph{tempered} irreducible unitary representations \cite[Sec.~25]{HarishChandra66}. See for instance \cite[Thm.~12.23]{Knapp1} and \cite[Thms.~1 and 2]{CowlingHaagerupHowe}, together with \cite[Rmk.~(b), p.103]{CowlingHaagerupHowe} for a proof of this.    To a first approximation, the goals of this paper  are to determine the tempered dual of a real reductive group as a topological space, and to compute its $K$-theory.

The tempered dual may be identified  with the (topological) space of irreducible representations of the  \emph{reduced $C^*$-algebra}  $C^*_r (G)$, which is the quotient of the full group $C^*$-algebra by the kernel of the regular representation \eqref{eq-regular-rep}.   Our precise goals are to determine  $C^*_r (G)$ up to Morita equivalence, and to compute its $K$-theory.

Studying the tempered dual as a topological space (and at the same time studying the reduced $C^*$-algebra)  is  rewarded in   spectacular fashion by a beautiful isomorphism statement in $K$-theory that was conjectured by A. Connes \cite[Conjecture~4.1]{RosenbergNeptun}  and G. Kasparov  \cite[Sec.~5, Conjecture~1] {KasparovICM}.   This \emph{Connes-Kasparov isomorphism} is now viewed as part of the more  general \emph{Baum-Connes conjecture} about the reduced group $C^*$-algebra of any locally compact group (including any discrete group).  See  \cite[\S 4]{BCH} or \cite{BCsurvey19} for a recent survey of the status of the conjecture. 

In this work we shall present the full details of a representation-theoretic proof of the Connes-Kasparov isomorphism for connected, linear, real reductive  groups.  
Such a proof   was announced in outline form only by A. Wassermann  in \cite{NoteWassermann}, following pioneering work by M. Penington and R. Plymen \cite{PeningtonPlymen} and A. Valette \cite{Valette84, Valette85}. Subsequently, V. Lafforgue  gave an entirely new proof using his work on the Baum-Connes conjecture \cite{LafforgueInventiones}.  In some places we shall follow Wassermann's outline, but elsewhere we shall follow a quite different route.

Lafforgue explains in \cite[\S 2]{LafforgueICM} that if $G$ is   of equal rank and has compact center, so that it possesses discrete series representations, then one can recover Harish-Chandra's classification of the discrete series in terms of Harish-Chandra parameters as a consequence of  the Connes-Kasparov isomorphism for $G$. The starting point is the observation that each discrete series is isolated in the tempered dual, and so is  detectable in $K$-theory.  Our approach will in effect use Harish-Chandra's classification, rather than provide an independent verification. But in place of that we shall  give an explicit answer to a natural question that arises from the classification.  Harish-Chandra's parameters are weights (for a chosen maximal torus of a maximal compact subgroup $K$ of $G$) that satisfy a nonsingularity condition.  The same parameters determine Dirac-type operators on the symmetric space $G/K$, and as R. Parthasarathy \cite{Parthasarathy} and Atiyah-Schmid \cite{AtiyahSchmid} explained, the associated discrete series can be constructed as the space of harmonic spinors for this operator. If Harish-Chandra's nonsingularity condition is dropped,  there is \emph{still} an associated Dirac operator. We shall show that this Dirac operator determines not a single representation of $G$ but a single \emph{connected component} of the tempered dual, and we shall describe this component in full detail, thereby in some sense completing Harish-Chandra's parametrization.

We shall reach this goal in the second paper.  Our starting point here is an earlier paper \cite{CCH1} that gave a detailed account of the structure of $C^*_r (G)$, in the form of a Paley-Wiener type theorem for the reduced group $C^*$-algebra. 
Here we shall go one step further and give a computation of the $K$-theory of the reduced $C^*$-algebra using 
a dichotomy, first observed by
Wassermann in \cite{NoteWassermann}, that invokes the Knapp-Stein theory of intertwining operators  \cite{KS1,KS2} to separate the components of the tempered dual with trivial $K$-theory from those with nontrivial $K$-theory.  The result is that   the $K$-theory of the tempered dual is carried by those components   that are \emph{essential}, which means  that every Knapp-Stein intertwiner for the component belongs to the Knapp-Stein $R$-group. See Definition~\ref{def-essential-component}.  Moreover each essential component contributes one free generator to $K$-theory.

On the other hand,  Connes and Kasparov conjectured that the $K$-theory is freely generated by the  indices of indecomposable Dirac-type operators on the symmetric space $G/K$, where $K$ is a maximal compact subgroup in $G$.  These Dirac operators are easily parametrized using more or less just  the representation theory of $K$.  In contrast the set of essential components of the tempered dual is far more mysterious.

Nevertheless, it is a remarkable fact that the $K$-theory generators associated to the essential components and the $K$-theory generators associated to indecomposable Dirac operators are, up to sign, \emph{exactly the same}.  We shall conclude the present paper by precisely formulating a slightly weaker correspondence  in what we call the \emph{matching theorem}; 
see Theorem~\ref{thm-matching}.  Finally, we shall  explain  how the Connes-Kasparov isomorphism follows quickly from the matching theorem.  We shall present two arguments---one that relies on fundamental results by Kasparov in $KK$-theory, and one that is purely representation theoretic.

In the second paper we shall complete our account of the $K$-theory of the tempered dual and the Connes-Kasparov isomorphism by proving the matching theorem, and more. In contrast to   earlier works on the Connes-Kasparov isomorphism, we shall use David Vogan's approach to the construction and classification of the tempered dual \cite{Voganbook}. We shall show that Vogan's theory leads to  a simple construction of all the essential components of the tempered dual, and only those components,  all at once. Although Dirac operators are not used in our construction, the data used to construct an essential component turns out to be exactly the same as the data used to construct an indecomposable Dirac operator.  From the point of view of representation theory, this is the fundamental result underlying the Connes-Kasparov isomorphism.

{\bf Notes on Terminology.} Throughout this paper and the sequel, by a \emph{real reductive group}, we shall always mean the group $G$ of real points in a connected complex reductive linear algebraic  group that is defined over $\R$.  See for instance \cite[Ch.~19]{MilneAlgGroups}.  The main reason for this assumption is to guarantee that the theory developed in Vogan's monograph \cite{Voganbook}, which will be crucial in the sequel, will  apply to $G$. 
But \emph{we shall also assume that $G$  is itself connected,} which will considerably simplify both the statements of theorems and their proofs. In this paper we shall often refer to the text \cite{Knapp1}, so let us note here that  our  groups are the same (up to isomorphism) as the  linear connected reductive groups in  \cite[\S I.1]{Knapp1}.
 
We shall use fraktur letters such as $\mathfrak{g}$, etc, to refer to the Lie algebras of Lie groups such as $G$, etc, and not to the complexifications of these Lie algebras. This is because we shall have little use for the complexified Lie algebras in this first paper. But in the second paper we shall use the complexifications extensively and we shall follow a different convention.  

When discussing Dirac operators we shall follow conventions appropriate to index theory on manifolds.  These are a bit different from the conventions in representation theory, where so-called Dirac cohomology is studied.  But in the second paper we shall switch and follow the Dirac operator conventions that are used in representation theory.

\section{Parabolic Induction and the Reduced Group C*-Algebra}
\label{sec-parabolic-induction}

In this section  we shall  review  the  description of the reduced $\Cs$-algebra of a real reductive group that was obtained  in \cite{CCH1} using  results in tempered representation theory due to Harish-Chandra, R. Langlands and others. Then, following Wassermann \cite{NoteWassermann}, we shall refine that description so as to determine the reduced group $C^*$-algebra up to Morita equivalence.

We shall fix, once and for all in this paper, a maximal compact subgroup $K{\subseteq} G$ and Cartan decomposition $\mathfrak{g} = \mathfrak{k}\oplus \mathfrak{s}$.  We shall also fix a maximal abelian subspace $\mathfrak{a}{\subseteq} \mathfrak{s}$ and a compatible Iwasawa decomposition $G=KAN$.
The associated minimal parabolic subgroup is $P_{\min} = MAN$, where $M$ is the centralizer of $A$ in $K$.  See \cite[Ch.~V]{Knapp1}.

Since we shall be working with convolution algebras we  shall also  fix a Haar measure on $G$, as well as a normalized Haar measure on $K$.
 
\subsection*{Parabolic Induction}

We begin by reviewing some essential points about parabolic induction \cite{CCH1}.     A standard   parabolic subgroup of $G$  is any closed subgroup $P$ of $G$ that includes $P_{\min}$. It   decomposes as a semi-direct product $P= L_PN_P$ of a Levi component $L_P$ that is mapped to itself by the Cartan involution and  the unipotent radical $N_P$. 
Furthermore, the Levi component $L_P$ is the product of the closed subgroup $M_P$ that is generated by all compact subgroups of $L_P$ and the split component $A_P$ of $L_P$. This  leads to a Langlands decomposition $ P =  M_PA_PN_P $.  See \cite[\S V.5]{Knapp1} or  \cite[Ch.~VII]{KnappBeyond}.
 
If $\pi$ is a unitary representation of $L_P$, then we may form the (\emph{unitarily}) \emph{parabolically induced representation} $\Ind_P^G \pi$, which is the unitary action of $G$ by left translation on the Hilbert space  completion of the  vector space of smooth functions
\begin{equation}
    \label{eq-induced-space}
\bigl \{\,  f \colon G \to H_\pi :  f(gman) = e^{ - \rho(\log a)} \pi(ma)^{-1} f(g) \, \bigr \} .
\end{equation}
The completion is taken with respect to the inner product 
\begin{equation}
    \label{eq-induced-inner-product}
\langle f_1,f_2 \rangle = \int _K \langle f_1(k) ,  f_2 (k)\rangle   \, dk ,
\end{equation}
and $\rho\in \mathfrak{a}_P^*$ is defined by 
\[
\rho (X) = \tfrac 12 \operatorname{Trace} \bigl ( \operatorname{ad}_X \colon \mathfrak{n}_P \to \mathfrak{n}_P\bigr ) .
\]
See  \cite[\S VII.1]{Knapp1}.

\begin{definition}
\label{def-principal-series}
Let  $P$ be a standard parabolic subgroup and let  $\sigma$ be an irreducible,  discrete series representation of $M_P$ (that is, an irreducible unitary representation of $M_P$, all of whose matrix coefficient functions are square-integrable)  and let $\varphi\in \mathfrak{a}^*_P$.  The formula 
\[
 \sigma {\otimes} e^{i \varphi} \colon  m a \longmapsto  e^{i \varphi ( \log a )} \sigma(m)
\]
defines a unitary representation of $L_P{=}M_PA_P$ on the Hilbert space of the representation $\sigma$.  The associated \emph{$(P,\sigma)$-principal series representation} of $G$ 
is the  unitary representation $\pi_{\sigma, \varphi}$ that is  obtained from $\sigma {\otimes} e^{i \varphi}$ by parabolic induction.   We shall denote by $\Ind_P^G H_\sigma{\otimes} \C_{i \varphi}$ the Hilbert space on which it acts. 
\end{definition}

\begin{remark}
The Langlands decomposition $P{=}MAN$, which is used in the definition above,  will be important at a number of places below.  It supplies each principal series of representations with a natural base point, where $\varphi =0$. This base point is not always available in other contexts, including  for instance that of $p$-adic groups, and the paper \cite{AfgoustidisAubert22} examines some of the difficulties that can arise   as a result.
\end{remark}

If we restrict the functions in \eqref{eq-induced-space} to $K$, then all the representations in the $(P,\sigma)$-principal series can be regarded as acting on the following common Hilbert space (and we shall mostly do so from now on):

\begin{definition}
\label{def-induced-representation-hilbert-space}
We shall denote by $\Ind_P^G H_\sigma$   the Hilbert space completion of the space of smooth functions 
\begin{equation}
    \label{eq-induced-space2}
\bigl \{\,  f \colon K \to H_\sigma :  f(kh) =   \sigma(h)^{-1} f(k) \,\,\forall h \in K\cap L\,, \forall k\in K\, \bigr \}
\end{equation}
in the inner product \eqref{eq-induced-inner-product}.
\end{definition}

\subsection*{Intertwining Operators}
\label{sec-decomp-C*rG}
The following concept gives us a first, large-scale view of the tempered dual:

\begin{definition}
\label{def-assoc} 
Let  $P_1=M_1A_1N_1$ and $P_2=M_2A_2N_2$ be standard  parabolic subgroups and let $\sigma_1$, $\sigma_2$ be irreducible square-integrable representations of $M_1$ and $M_2$ respectively. The pairs $(P_1,\sigma_1)$ and $(P_2,\sigma_2)$ are \emph{associate} if there exists an element $k\in K $ such that 
\[
\Ad_k  [ L_{P_1}  ] =L_{P_2} \qquad\text{and}\qquad\Ad_k^*\sigma_1\simeq \sigma_2.
\]
We shall call an equivalence class of pairs under this relation an \emph{associate class} and use the notation $[P,\sigma]$.
\end{definition}

The    theorem below summarizes some of the important results of Harish-Chandra, Langlands and others, and it is the foundation for our study of the tempered dual.

\begin{theorem}
\label{thm-phil-cusp-forms}
The tempered dual  admits a disjoint union decomposition
\begin{equation*}
\Gr=\bigsqcup_{[P,\sigma]}\widehat{G}_{P,\sigma}
\end{equation*} 
as a topological space, where $\widehat{G}_{P,\sigma}$ consists of the irreducible components of the $(P,\sigma)$-principal series representations and the union is indexed by associate classes. Each  part $\widehat G_{P,\sigma}$ is a connected and open  subset of $\Gr$ \textup{(}and it follows that each part is also closed\textup{)}.
\end{theorem}

As for the proof, it is explained by Lipsman in  \cite{Lipsman70} that for semisimple groups with finite center, the topology on each minimal principal series component of the tempered dual coincides with the ``natural topology'' inherited from 
the space of continuous parameters of the principal series, and in particular each principal series is connected. The general case of Lipsman's result (covering all reductive groups and all cuspidal principal series)  may be proved in the same way. But in any case, a complete proof of this and all the other assertions in Theorem~\ref{thm-phil-cusp-forms} is given in \cite{CCH1}. Theorem~6.8 in \cite{CCH1}, which is reproduced below as Theorem~\ref{thm-CsrG}, gives a nearly complete description of the reduced $C^*$-algebra of $G$, using which the spectrum of the reduced $C^*$-algebra is easily determined.  But that spectrum is precisely the tempered dual, as we noted in the Introduction. So for instance,  the fact that each principal series component is open and closed in the tempered dual follows from the direct sum decomposition in Theorem~6.8 of \cite{CCH1}.

To   probe the   equivalences within a single principal series family, as well as the possible reducibility of the representations within that family, one studies intertwining operators.

\begin{definition}
\label{def-Weyl-grp} 
The \emph{intertwining  group} associated with a pair $(P,\sigma)$ as above is the finite group 
\[
W_\sigma= \bigl \{w \in N_K(L_P) : \operatorname{Ad}_w ^* \sigma \simeq \sigma\, \bigr \} \big /  K\cap L_P.
\]
\end{definition}

The theory of intertwining operators, due to Knapp and Stein \cite{KS1,KS2}, associates to each $w\in W_\sigma$, and each $\varphi \in \mathfrak{a}^*_P$, a unitary operator 
\begin{equation}
\label{def-Uw}U_{w,\varphi}: \Ind_P^G H_\sigma\longrightarrow\Ind_P^G  H_\sigma, 
\end{equation}
that intertwines the principal series representations $\pi_{\sigma,\varphi}$ and $\pi_{\sigma, w(\varphi)}$, so that if $g\in G$, then 
\begin{equation}
    \label{eq-knapp-stein-intertwining-property}
U_{w,\varphi} \pi_{\sigma,\varphi}(g)  = \pi_{\sigma, w(\varphi)}(g) U_{w,\varphi} .
\end{equation}
Here $w(\varphi)(X) = \varphi (\Ad_{w^{-1}}(X))$.  The operators $U_{w,\varphi}$ vary strongly-contin\-uously with $\varphi\in \mathfrak{a}^*_P$.

  The construction of $U_{w,\varphi}$ involves a choice of unitary equivalence of representations $\Ad^*_w\sigma\simeq \sigma$.  By Schur's lemma the choice is unique up to a multiplicative scalar of modulus one, but it is not absolutely unique.  This leads to a cocycle relation
\begin{equation}
\label{eq-cocycle-relation}
U_{w_1,w_2(\varphi)}U_{w_2,\varphi} = c(w_1,w_2)U_{w_1w_2,\varphi}\qquad \forall \varphi \in \mathfrak{a}_P^*
\end{equation}
with $| c(w_1,w_2)| =1$. 

In fact the equivalences  $\Ad^*_w\sigma\simeq \sigma$ may be chosen so that  $c(w_1,w_2) =1$ for all $w_1,w_2\in W_\sigma$.  This is not trivial, but it is important for what follows, so let us describe the method.

\begin{definition} 
\label{def-w-prime-group}
\cite[\S 13]{KS2}
Denote by  $W_\sigma'\triangleleft W_\sigma$ the normal subgroup
\[
W'_\sigma=\bigl \{  w'\in W_\sigma \: : \:\text{The intertwiner $U_{w',0}$ acts as a scalar on $\Ind_P^G H_\sigma$}\bigr \}.
\]
Denote by $R_\sigma$ the quotient group 
$W_\sigma/ W_\sigma'$.
\end{definition}

For each $w'\in W_\sigma'$ we can choose an equivalence $\Ad^*_{w'}\sigma\simeq \sigma$  so that in fact each $U_{w',0}$ acts as the identity operator  on $\Ind_P^G H_\sigma$, and then having done so we obtain $c(w'_1,w'_2) = 1$ for all $w'_1,w'_2\in W'_\sigma$.

 \begin{theorem}
\label{th-R-group1}\cite[Thm. 13.4]{KS2}
The quotient group homomorphism from $W_\sigma$ to  $R_\sigma$ splits, and the intertwining group $W_\sigma$ therefore admits a semi-direct product decomposition $W_\sigma=W'_\sigma\rtimes R_\sigma$.
 \end{theorem}

\begin{theorem}[{\cite[\S13 and \S15]{KS2} and \cite[\S 6]{Knapp_commut}}]
\label{th-R-group2-prelim}
The $R$-group  is abelian, and indeed a finite product of groups of order two.  Moreover there is a splitting, as in the  previous theorem, so that the associated intertwiners $U_{w,0}$ \textup{(}$w\in R_\sigma$\textup{)} also pairwise commute.
\end{theorem}   

Returning to the cocycle issue, fix a splitting as in   Theorem~\ref{th-R-group2-prelim}. Since $R_\sigma$ is a direct product of groups of order two, and since the associated Knapp-Stein intertwiners commute, we can  certainly choose equivalences $\Ad^*_w\sigma\simeq \sigma$ for $w\in R_\sigma$ so that $c(w_1,w_2)=1$ in \eqref{eq-cocycle-relation}, for all $w_1,w_2\in R_\sigma$.  If we now make further adjustments by scalars so that $U_{w,0}U_{w',0} = U_{ww',0}$ for all $w\in R_\sigma$ and all $w'\in W'_\sigma$, then we shall obtain $c(w_1,w_2) \equiv 1$ for all $w_1,w_2\in W_\sigma$, as required. We shall use these adjusted Knapp-Stein intertwining operators from now on.

The Knapp-Stein intertwining operators completely account for the decomposition of   principal series representations into irreducible representations, and for equivalences among these irreducible summands.  We refer to  \cite[\S 6]{CCH1} for a summary   that is adapted to our purposes; the same information will be encoded in  the description of the reduced $C^*$-algebra in Theorem~\ref{thm-CsrG} below.

\subsection*{The Reduced Group C*-Algebra}

Let $P=M_PA_PN_P$ be a   parabolic subgroup of $G$, and let  $\sigma$ be an irreducible square-integrable representation of $M_P$. The $(P,
\sigma)$-principal series representations are tempered, and they therefore determine representations $\pi_{\sigma,\varphi}$ of the reduced group $C^*$-algebra using  formula \eqref{eq-rep-of-the-group-algebra}.
We now introduce the $C^*$-algebra  
\begin{equation}
\label{def-component-c-star-algebra}
 C_0(\mathfrak{a}_P^*,\mathfrak{K}(\Ind_P^G H_\sigma))
\end{equation}
of norm-continuous functions,   vanishing in norm at infinity, from the locally compact space  $\mathfrak{a}_P^*$  to the $C^*$-algebra of compact operators on the Hilbert space  $\Ind_P^G H_\sigma$.  

\begin{proposition}[{\cite[Cor.~4.12]{CCH1}}]
\label{prop-component-of-cstar-of-g}
There is a \textup{(}unique\textup{)} $C^*$-algebra homomorphism 
\[
\pi_\sigma \colon C^*_r (G) 
\longrightarrow 
 C_0(\mathfrak{a}_P^*,\mathfrak{K}(\Ind_P^G H_\sigma))
 \]
 such that $ \pi_\sigma (f) (\varphi) = \pi_{\sigma,\varphi}(f)$
 for every $f\in C_c^\infty (G)$ and every $\varphi \in \mathfrak{a}^*_P$.
\end{proposition}

\begin{remark}
 In \cite{CCH1}, the right-hand side is described in terms of functions on $\widehat{A}_P$, which identifies with $\mathfrak{a}^*_P$ through the exponential map as in Definition \ref{def-principal-series}.
\end{remark}

There is an action of the intertwining group $W_\sigma$   on the $C^*$-algebra \eqref{def-component-c-star-algebra} that is characterized by the formula
 \[
 w(f)(w(\varphi)) = U_{w, \varphi} f( \varphi) U_{w, \varphi}^* 
 \]
 for all $w\in W_\sigma$ and all $\varphi\in \mathfrak{a}_P^*$. It follows from the intertwining property \eqref{eq-knapp-stein-intertwining-property} that the image of the morphism $\pi_\sigma$ in Proposition~\ref{prop-component-of-cstar-of-g} is fixed pointwise by this action of $W_\sigma$.  The description of  the reduced $\Cs$-algebra given in \cite{CCH1} is as follows:
 
\begin{theorem}[{\cite[Thm. 6.8]{CCH1}}]
\label{thm-CsrG}
The morphisms in Proposition~\ref{prop-component-of-cstar-of-g} combine to give  an {isomorphism}  of $\Cs$-algebras
\begin{equation*}
\Csr(G)\overset{\cong}{\longrightarrow}\bigoplus _{[P,\sigma]}C_0\bigl ( \mathfrak{a}_P^*, \mathfrak{K}(\Ind_P^G H_\sigma  )\bigr ) ^{W_\sigma}.
\end{equation*}
The direct sum is the $C_0$-direct sum of $C^*$-algebras over a choice of representatives of the associate classes $[P,\sigma]$.
\end{theorem}

\subsection*{The Principal Series as an  Equivariant  Bundle}
\label{subsec-equivariant-bundle}
In this section we shall show, following Wassermann \cite{NoteWassermann}, that the $W_\sigma$-action that is used to define the fixed-point algebra  Theorem~\ref{thm-CsrG} can be replaced by a much simpler action in a way that does not change those fixed-point algebra, up to $*$-isomorphism.

Form the trivial  bundle of Hilbert spaces with fiber $\Ind_P^G H_\sigma$ over the locally compact space  $\mathfrak{a}_P^*$ . The Knapp-Stein intertwiners determine an action on this bundle,
\begin{equation}
\label{eq-w-equivariant-bundle}
W_\sigma \times ( \mathfrak{a}_P^* \times \Ind_P^G H_\sigma ) \longrightarrow  ( \mathfrak{a}_P^* \times  \Ind_P^G H_\sigma ),
\end{equation}
  via the formula 
\begin{equation}
\label{eq-w-equivariant-bundle-formula1}
w \cdot (\varphi, v) = (w(\varphi), U_{w, \varphi} v).
\end{equation}
We shall now give a simpler description, up to isomorphism, of this $W_\sigma$-equi\-variant Hilbert space bundle.

\begin{remark}
The bundle $\mathfrak{a}_P^* {\times} \Ind_P^G H_\sigma$ is infinite-dimensional, but  it decomposes canonically as the orthogonal Hilbert direct sum of its  finite-dimensional $K$-isotypic components, which are finite-dimensional $W_\sigma$-equi\-variant bundles in their own right.  One could, if one preferred, work with these finite-dimensional bundles.
\end{remark}

Define a second $W_\sigma$-action on the   bundle  $\mathfrak{a}_P^*{\times} \Ind _P^G H_\sigma$ by the formula 
\begin{equation}
\label{eq-w-equivariant-bundle-formula2}
w \cdot (\varphi, v) = (w(\varphi), U_{w, 0} v).
\end{equation}

\begin{proposition}[\emph{c.f.} {\cite[Cor.~5]{NoteWassermann}}]
\label{prop-pullback-bundle}
The two $W_\sigma$-equivariant bundle structures on    $\mathfrak{a}_P^*{\times} \Ind _P^G H_\sigma$  defined by  the two actions \eqref{eq-w-equivariant-bundle-formula1} and \eqref{eq-w-equivariant-bundle-formula2} are unitarily equivalent.
\end{proposition}

\begin{proof}
The single-point subset   $\{0\}\subseteq \mathfrak{a}_P^* $  is a $W_\sigma$-equivariant deformation retract.  
It therefore follows from elementary vector bundle theory  that any two $W_\sigma$-equivariant bundles over $\mathfrak{a}_P^* $ whose fibers over  $0$  are unitarily equivalent as representations of $W_\sigma$ are in fact unitarily equivariantly isomorphic as bundles.  
\end{proof}

\begin{corollary}
\label{cor-isomorphic-fixed-point-algebras}
The $W_\sigma$-actions on  the $C^*$-algebra $C_0(\mathfrak{a}_P^*  , \Compact( \Ind _P^G H_\sigma ) )$  defined by the formulas \eqref{eq-w-equivariant-bundle-formula1} and \eqref{eq-w-equivariant-bundle-formula2} are conjugate by a $C^*$-algebra automorphism. In particular, the corresponding fixed-point $C^*$-subalgebras are isomorphic. \qed
\end{corollary}

\subsection*{The Reduced C*-Algebra up to Morita Equivalence}
\label{subsec-Morita}

In this section we shall construct a Morita equivalence between each summand in the decomposition of  Theorem~\ref{thm-CsrG} and a still more elementary   $\Cs$-algebra.  We shall continue to  follow  Wassermann \cite{NoteWassermann} closely.  

A complete treatment of the Morita equivalence has  appeared very recently in \cite{AfgoustidisAubert22}, where many situations involving $p$-adic groups are also considered.   For the sake of completeness we shall nonetheless present the argument below, although in view of the existence of \cite{AfgoustidisAubert22} we shall take the liberty of omitting some details.

We shall use Corollary~\ref{cor-isomorphic-fixed-point-algebras}, and throughout this subsection we shall work with the action of $W_\sigma$ on the $C^*$-algebra 	$C_0(\mathfrak{a}_P^*   , \Compact( \Ind _P^G H_\sigma ) )$ that is derived from \eqref{eq-w-equivariant-bundle-formula2}.
We shall determine the fixed-point $C^*$-subalgebra up to Morita equivalence.

The group $R_\sigma$ acts on the $C^*$-algebra $C_0\bigl (\mathfrak{a}_P^*/W'_\sigma, \mathfrak{K}(\Ind_P^G H_\sigma)\bigr )$ via the formula 
\[
(r \cdot f )([\varphi]) = U_{w,0}f([w^{-1}(\varphi)]) U_{w,0}^*,
\]
where $w$ is any preimage in $W_\sigma$ of  $r\in R_\sigma$.  The morphism 
\begin{equation}
    \label{eq-first-cstar-isomorphism-using-r-group}
	C_0\bigl(\mathfrak{a}_P^*   , \Compact( \Ind _P^G H_\sigma )\bigr ) ^{W_\sigma}
	\stackrel\cong \longrightarrow 
	C_0\bigl (\mathfrak{a}_P^*/W'_\sigma, \mathfrak{K}(\Ind_P^G H_\sigma)\bigr ) ^{R_\sigma}
\end{equation}
defined by the formula 
\[
f \longmapsto \Bigl [ [\varphi]\mapsto f (\varphi)\Bigr ]
\]
is an isomorphism of $C^*$-algebras. We shall therefore concentrate on the $R_\sigma$-fixed point $C^*$-algebra.

Now form the space $\Compact \bigl (\Ind_P^G H_\sigma,\ell^2R_\sigma
\bigr )$ of compact Hilbert space operators from $\Ind_P^G H_\sigma$ into the finite-dimensional Hilbert space $\ell ^2 R_\sigma$. Use the action of $R_\sigma$ on $\Ind_P^G H$, along with the left-translation  action of $R_\sigma$ on $\ell^2 R_\sigma$,
to define an $R_\sigma$-action on $\Compact \bigl (\Ind_P^G H_\sigma,\ell^2R_\sigma
\bigr )$.

Finally, form the Banach  space
\begin{equation}
\label{eq-morita-equivalence-bimodule}
C_0\bigl (\mathfrak{a}_P^*/W_\sigma' ,\Compact  (\Ind_P^G H_\sigma,\ell^2R_\sigma
 )\bigr )^{R_\sigma} .
\end{equation}
It carries commuting actions of  the $C^*$-algebra   $C_0 (\mathfrak{a}^*_{P}/W'_\sigma, \Compact (\Ind_P^G H_\sigma ))^{R_\sigma}$ on the right, by pointwise composition, and of    $C_0 (\mathfrak{a}^*_{P}/W'_\sigma,\Compact (\ell^2R_\sigma ))^{R_\sigma}$ on the left.  

\begin{theorem}[\emph{c.f.} {\cite[Cor.~7]{NoteWassermann}}]
\label{th-main-Morita}
For each associate class $[P,\sigma]$, the bimodule 
\[
C_0\bigl(\mathfrak{a}^*_{P}/W'_\sigma ,\Compact(\Ind_P^G H_\sigma,\ell^2R_\sigma)\bigr)^{R_\sigma}
\] 
implements a strong Morita equivalence  
\[
C_0\bigl (\mathfrak{a}^*_{P}/W'_\sigma,\Compact (\ell^2R_\sigma )\bigr)^{R_\sigma}
	\underset{\text{\rm  Morita}}\simeq 
 C_0\bigl (\mathfrak{a}^*_{P}/W'_\sigma, \Compact (\Ind_P^G H_\sigma )\bigr)^{R_\sigma}.
\]
\end{theorem}

Let us quickly recall  the $C^*$-algebraic  concept of Morita equivalence, which includes analytic requirements that are obviously absent from the purely algebraic theory (among other things, they help extending the reach of the theory to non-unital $C^*$-algebras).  A succinct formulation is as follows: an equivalence $A$-$B$-bimodule must have the form $p Cp^\perp$, where
 \begin{enumerate}[\rm (i)]
     \item $C$ is a $C^*$-algebra and $p$ is a projection in the multiplier algebra  of $C$ \cite[\S 3.12]{Pedersen79};
     \item $C^*$-algebra isomorphisms  are  provided between $pCp$ and $A$, and between $p^\perp C p^\perp$ and $B$ (where $p^\perp = 1{-}p$); and 
     \item $pCp^\perp C p$ and $p^\perp C p C p^\perp$ are dense in $pCp$ and $p^\perp C p^\perp$, respectively.
     
 \end{enumerate}
See \cite{RaeWillMorita}. In the present case,  $C$ will be   the $C^*$-algebra of $R_\sigma$-fixed functions of class $C_0$ from $\mathfrak{a}_P^* / W'_\sigma$ to the $C^*$-algebra of compact operators on the direct sum Hilbert space $\Ind_P^G H_\sigma \oplus \ell ^2 R_\sigma$, and $p$ will be the projection onto the second Hilbert space direct summand, and then  $pCp^\perp$ will be \eqref{eq-morita-equivalence-bimodule}.

\begin{lemma}
\label{lem-Morita-general}
 Let $\Gamma$ be a finite  group acting properly on a locally compact Hausdorff  space $X$,  and let $H_1$ and $H_2$ be Hilbert spaces equipped with unitary representations of $\Gamma$. If for every $x\in X$,  $H_{1}$ and $H_{2}$ are weakly equivalent representations  of the stabilizer subgroup  $\Gamma_x$ \textup{(}that is, each is contained in a multiple of the other\textup{)}, then the bimodule \[C_0\bigl ( X,\Compact(H_2,H_1)\bigr )^\Gamma\] implements a   Morita equivalence of $C^*$-algebras
\[
C_0\bigl (X,\Compact (H_1)\bigr )^\Gamma\underset{\text{\rm  Morita}}\simeq 
C_0\bigl(X,\Compact (H_2)\bigr )^\Gamma.
\]
\end{lemma}

\begin{proof}
Denote the two $C^*$-algebras in the statement of the proposition by $A$ and $B$, and the $A$-$B$-bimodule by $E$. We need to show that the sets
\[
\{\, f  g^* : f,g\in E\, \}  
\quad\text{and} \quad 
\{\, f^* g : f,g\in E\, \}    
\]
span dense ideals in $A$ and $B$, respectively.

If an ideal in a $C^*$-algebra is \emph{not} dense, then there is an irreducible representation of the $C^*$-algebra that vanishes on the ideal \cite[Ch.~2]{DixmierEnglish}.   So to prove density in $A$  we need only show that for every irreducible representation  of  $A$ there is some element  $  f{\in} {E}$ such that the representation is nonzero on  $ff^*\in A$.

Each irreducible representation of $A$ must factor through evaluation of functions $A$ in  at some point $x{\in} X$, since by Schur's lemma  all functions in $C_0(X)^\Gamma$ must act on the representation space as scalar multiples of the identity.
But the image of $A$ under evaluation at $x{\in} X$  is $\Compact (H_{1})^{\Gamma_x}$, and its irreducible representations   are precisely the nonzero  $\Gamma_x$-isotypical subspaces $H_1^\rho$ of $H_1$ ($\rho \in \widehat \Gamma_x$).

By hypothesis, the  isotypical subspace $H_{2}^\rho$   is nonzero, and hence there is a   nonzero $\Gamma_x$-equivariant compact operator $T\colon H_{2}\to H_{1}$ whose range lies in $H^\rho_{1}$, and there is a function  $f\in E$ whose value at $x$ is $T$. But now the value of  $ff^*\in A$ at $x$ is equal to $TT^*$, which is nonzero.
\end{proof}

We shall need to combine the simple computation above with the following more substantial result from the Knapp-Stein theory:
 
\begin{theorem}[{\cite[Theorem~14.43]{Knapp1}}]
\label{thm-r-group-linear-independence}
The intertwining operators 
\[
U_{w,0}\colon \Ind_P^G H_\sigma \longrightarrow \Ind_P^G H_\sigma  \qquad ( w \in R_\sigma) 
\]
 are linearly independent of one another. 
\end{theorem}

\begin{corollary} 
\label{lem-regular-repn}
The representation of $R_\sigma$ on $\Ind_P^G H_\sigma$  includes a copy of every irreducible representation of $R_\sigma$.
\end{corollary} 

\begin{proof} The representation of $R_\sigma$ on $\Ind_P^G H_\sigma$ determines a representation of the complex group algebra of $R_\sigma$, and  Theorem \ref{thm-r-group-linear-independence} implies that this algebra representation   is faithful. That is, every element of the group algebra acts as a nonzero operator. So each of the isotypical projections associated to the irreducible representations of $R_\sigma$ acts as a nonzero operator on $\Ind_P^G H_\sigma$.  
\end{proof}

\begin{proof}[Proof of Theorem~\ref{th-main-Morita}]
It follows from the corollary above 
that $\Ind_P^G H_\sigma$  includes a copy of every irreducible representation of every subgroup of $R_\sigma$, and certainly the same is true of $\ell^2 R_\sigma$. So Lemma~\ref{lem-Morita-general} applies with $X=\mathfrak{a}^*_{P}/W'_\sigma$ and $\Gamma=R_\sigma$.
\end{proof}

We shall conclude this section by showing how the statement of  Theorem~\ref{th-main-Morita}  can be streamlined using some standard $C^*$-algebra language (although we shall not use this language in what follows).    

Let  $\Gamma$ be a finite group. Denote by  $\lambda$ and $\rho$    the actions of $\Gamma$ on $\Compact (\ell^2 \Gamma)$ associated with the left and right regular representations, respectively.  In addition, if $\gamma \in \Gamma$, then denote by  $e_\gamma$   the rank-one projection onto  the functions  in $\ell^2\Gamma$  that are supported on $\gamma$.  If  $A$ is any $\Cs$-algebra with a $\Gamma$-action, then we denote by $A \rtimes \Gamma $ the crossed product $C^*$-algebra.

\begin{lemma}[{\cite[Prop.~4.3]{Rieffel80}}]
\label{lem-Rieffel}
The linear map
\[
 A \rtimes\Gamma   \longrightarrow (A\otimes \Compact(\ell^2 \Gamma))^{\Gamma,\lambda}
 \]
  defined by
\[
a\mapsto \textstyle \sum _{\gamma\in \Gamma} \gamma(a) \otimes e_\gamma \quad \text{and} \quad \gamma \mapsto 1 \otimes \rho (\gamma), 
\]
is an isomorphism of $C^*$-algebras. 
  \qed
\end{lemma}

Combining Lemma~\ref{lem-Rieffel}  with    Theorem~\ref{th-main-Morita}, we obtain for any component $[P,\sigma]$ of the tempered dual a  Morita equivalence
\begin{equation}
\label{eq-Morita-summand}
 C_0\bigl (\mathfrak{a}^* , \Compact  (\Ind_P^G H_\sigma  )\bigr)^{W_\sigma}\underset{\text{\rm Morita}}{\simeq} C_0(\mathfrak{a}^*_P/W'_\sigma) \rtimes {R_\sigma}.
\end{equation}

Assembling the summands using the isomorphism of Theorem~\ref{thm-CsrG}, we obtain the following picture of the reduced $\Cs$-algebra up to Morita equivalence, due to Wassermann \cite{NoteWassermann}.

\begin{theorem}[{\cite[Thm. 8]{NoteWassermann}}]
\label{thm-CsrG-Morita}
There is a Morita equivalence  of $\Cs$-algebras
\begin{equation*}
\Csr(G)\underset{\text{\rm Morita}}{\simeq} \bigoplus _{[P,\sigma]}C_0(\mathfrak{a}^*_P /W'_\sigma ) \rtimes {R_\sigma},
\end{equation*}
where the   sum is    over    representatives of the  associate classes    $[P,\sigma]$.
\end{theorem}

 \section{Further Information about the Knapp-Stein  Intertwining Groups}
 \label{sec-knapp-stein-vogan-r-groups}
The results in the preceding section give an account of the structure of $C^*_r(G)$ up to Morita equivalence in terms of the intertwining groups $W_\sigma$ and their semi-direct product decompositions $W_\sigma = W_\sigma ' \rtimes R_\sigma$.  
In this section we shall summarize the additional facts about these decompositions that we shall need to complete the computations in this paper.

\subsection*{The W'-Group}

We defined $W'_\sigma$ using the action of the Knapp-Stein intertwining operators on the representations $\pi_{\sigma, 0}$.  An important result is that $W'_\sigma$  is also the Weyl group of a root system:

\begin{theorem}[{\cite[\S13 and \S15]{KS2} and \cite[\S 6]{Knapp_commut}}]
\label{th-W-prime-group}
The   subgroup $W'_\sigma\triangleleft W_\sigma$ is the Weyl group of a \textup{(}possibly non-reduced\textup{)} root system $\Delta'_\sigma $ spanning a subspace\footnote{To be precise, there is an isomorphism from  $W'_\sigma$ to the  Weyl group of a root system $\Delta_\sigma'$ spanning a subspace, and the isomorphism gives the action of $W'_\sigma$ on that subspace. There is a complementary subspace on which the action of $W_\sigma'$ is trivial.} of $\mathfrak{a}_P^*$.  The action of the group $W_\sigma$ on $\mathfrak{a}_P^*$ permutes the roots in $\Delta'_\sigma$.
 \end{theorem}

\begin{definition}
We shall denote by 
  \[
  \mathfrak{a}^*_{\sigma,+}\subseteq \mathfrak{a}^*_P
  \]
   the (closed) dominant Weyl chamber in $\mathfrak{a}^*_P$ associated to some fixed system of positive roots $\Delta_{\sigma, +}'\subseteq \Delta_\sigma'$.  
   \end{definition}

See   \cite[Ch.~XIV, Sec.~9]{Knapp1} for the definition of the root system $\Delta'_\sigma$. One important consequence of Theorem~\ref{th-W-prime-group} for us will be that the quotient $\mathfrak{a}^*_P/W_\sigma'$ may be identified with the dominant chamber  $\mathfrak{a}^*_{\sigma,+}$  (in more detail, we shall use the fact that the projection map from the closed dominant chamber to $\mathfrak{a}^*_P/W_\sigma'$ is a homeomorphism).

\subsection*{The R-Group}

 Using the system of positive roots,  Knapp and Stein define $R_\sigma$ as a subgroup of $W_\sigma$, as follows:

\begin{definition}
\label{def-knapp-stein-r-group}
 The Knapp-Stein $R$-group $R_\sigma\subseteq W_\sigma$ is the  subgroup consisting of those elements that permute the positive roots  $\Delta'_{\sigma,+}\subseteq \Delta'_\sigma$  among themselves.
\end{definition}

This is consistent with our previous terminology: the subgroup $R_\sigma$ normalizes $W'_\sigma$, and since $W_\sigma$ acts by permutations   on the Weyl chambers in $\mathfrak{a}^*_P$ for the root system $\Delta'_\sigma$, while $W'_\sigma$ acts on the chambers simply-transitively, there is a semi-direct product decomposition $W _\sigma =  W'_\sigma\rtimes R_\sigma$.

Since $R_\sigma$ permutes the positive roots  among themselves, the action of $R_\sigma$ on $\mathfrak{a}^*_P $ restricts to an action 
\[
R_\sigma \times \mathfrak{a}^*_{\sigma,+}\longrightarrow \mathfrak{a}^*_{\sigma,+}
.
\]
We shall use this action in the next section.

By a \emph{reflection} of $\mathfrak{a}^*_P$ we shall mean an isometric  involution of $\mathfrak{a}^*_P$ with a one-dimensional ${-}1$-eigenspace.  Two reflections are \emph{orthogonal} if their ${-}1$-eigenspaces are orthogonal.

\begin{theorem}[{\cite[\S13 and \S15]{KS2} and \cite[\S 6]{Knapp_commut}}]
\label{th-R-group2}
The $R$-group associated to every  associate class $[P,\sigma]$  is  a finite product of groups of order two that act by pairwise orthogonal reflections on $\mathfrak{a}^*_P$.  
\end{theorem}   

 Finally, we shall need   the size of the group $R_\sigma$ in a crucial special case.
 
 \begin{definition} 
 \label{def-a-max}
Denote by $\mathfrak{a}_{\max}$ the split part of a maximally compact Cartan subalgebra of $\mathfrak{g}= \mathfrak{k} \oplus \mathfrak{s}$.  Thus $\mathfrak{a}_{\max}$ is the fixed part in $\mathfrak{s}$ of the action of a maximal torus in $K$.
\end{definition}

The space $\mathfrak{a}_{\max}$ is unique up to conjugation by elements of $K$, and so its dimension $\dim(\mathfrak{a}_{\max})$ is independent of any choices.

\begin{lemma}
$
 \dim(\mathfrak{a}_{\max}) \equiv \dim (G/K) \pmod 2.
$
 \end{lemma}

\begin{proof}
The action of the maximal torus associated to $\mathfrak{a}_{\max}$ on $\mathfrak{s}\ominus \mathfrak{a}_{\max}$ has no nonzero fixed vectors. Since every non-trivial irreducible representation of the torus has  dimension $2$, it follows that $\mathfrak{s}\ominus \mathfrak{a}_{\max}$ is even-dimensional.
\end{proof}

  \begin{theorem}
\label{th-R-group3}
If $[P,\sigma]$ is an associate class, and if $W'_\sigma= \{ e\}$, then the group $R_\sigma$ is generated by $\dim(\mathfrak{a}_P) {-} \dim(\mathfrak{a}_{\max})$  pairwise orthogonal reflections on $\mathfrak{a}$. \end{theorem}

We refer the reader to \cite{MatchingTheorem} for a proof using an alternative approach to the $R$-group due to Vogan \cite[Sec.~ 4.3]{Voganbook} (which seems to be much better suited to the problem of computing $R_\sigma$ in the essential case).

\section{K-Theory of the Reduced C*-Algebra}
 
In this section we shall compute the $K$-theory \cite{Rordam} of  $C^*_r (G)$ as an abstract abelian group. Since it is a basic feature of $K$-theory  that for any family of $C^*$-algebras $\{ A_\alpha\}$ the natural map 
\[
\bigoplus_\alpha K_*(A_\alpha) \longrightarrow K_*(\bigoplus_\alpha A_\alpha)
\]
is an isomorphism, we can and shall   focus on the individual fixed-point algebras \[C_0 (\mathfrak{a}^*_P, \mathfrak{K} (\Ind
_P^G H_\sigma) )^{W_\sigma}\] 
that make up the reduced group $C^*$-algebra. We have seen that these are Morita equivalent to the $C^*$-algebras
\[C_0\bigl (\mathfrak{a}^*_{P}/W'_\sigma, \Compact (\ell^2 R_\sigma )\bigr)^{R_\sigma}.
\]
 Since $K$-theory is a Morita invariant it suffices to study the latter.
 
 The computations below are very simple from a $K$-theoretic point of view, but they require the difficult results about the $R$-group that we surveyed in the last section.

\subsection*{Essential and Inessential Components}
\label{subsec-essential-and-inessential}
The results in this section are due to Wassermann \cite{NoteWassermann}. 
We start from the following partition of the set of associate classes $[P,\sigma]$.

 \begin{definition}
 \label{def-essential-component}
 An associate class $[P,\sigma]$  is called \emph{essential} if the normal subgroup $W'_\sigma \triangleleft W_\sigma$ is   trivial.  Otherwise $[P,\sigma]$  is called \emph{inessential}.
 \end{definition}

\begin{theorem}
If $[P,\sigma]$ is inessential, then 
$K_*   \bigl (C_0\bigl (\mathfrak{a}^*_{P}/W'_\sigma, \Compact (\ell^2 R_\sigma )\bigr)^{R_\sigma}\bigr )   = 0 $.
\end{theorem}

\begin{proof}
Identify the quotient  $\mathfrak{a}^*_P / W'_\sigma$   with the dominant Weyl chamber   $\mathfrak{a}^*_{\sigma, +}\subseteq \mathfrak{a}^*_P$.  The half-sum of the positive roots is a nonzero vector $\rho$ in the chamber that   is fixed under the action of $R_\sigma$.  The translations by nonnegative  multiples of  $\rho$  map $\mathfrak{a}_{\sigma, +}^*$ into itself and  give an $R_\sigma$-equivariant homotopy between  the identity morphism on the $C^*$-algebra $C_0 (\mathfrak{a}^*_{\sigma ,+} , \Compact (\ell^2 R_\sigma  ) )$ and the zero morphism.  So the $R_\sigma$-fixed-point algebra   is homotopy equivalent to zero. 
\end{proof}

The essential components have nonzero $K$-theory, and their treatment requires  more of the  $R$-group results from Section~\ref{sec-knapp-stein-vogan-r-groups}.

\begin{theorem}
\label{thm-k-theory-0f-essential-component}
If $[P,\sigma]$ is essential, then 
$K_*  \bigl (C_0\bigl (\mathfrak{a}^*_{P} , \Compact (\ell^2 R_\sigma )\bigr)^{R_\sigma} \bigr )  $ is a free abelian group on one generator, which lies  in degree $ \dim (G/K)\pmod 2$. 
\end{theorem}

Actually, for the sake of a later calculation we shall make a more precise statement directly in terms of the $K$-theory of  $C_0(\mathfrak{a}_P^*, \Compact (\Ind_P^G H_\sigma ) )$.  The assumption that $[P,\sigma]$ is essential  implies that the group $R_\sigma$ decomposes as a direct product
\[
R_\sigma \cong \underset{\text{$q$ times}}{\underbrace{\Z_2 \times \cdots \times \Z_2}},
\]
where $q {=}  \dim(\mathfrak{a}_P) {-} \dim ( \mathfrak{a}_{\max}  ) $; see Theorem~\ref{th-R-group3}. 
The  generators of the   factors act on $\mathfrak{a}_P^*$ as pairwise orthogonal  reflections and the fixed subspace 
\[
\mathfrak{a}_P^{*,R_\sigma}\subseteq \mathfrak{a}_P^*
\]
for the action of $R_\sigma$ has dimension $d{=}\dim(\mathfrak{a}_{\max})$.   

It follows from Theorem~\ref{thm-r-group-linear-independence} and Proposition~\ref{prop-pullback-bundle} that for   $\varphi\in \mathfrak{a}_P^{*,R_\sigma}$ the elements of the group $R_\sigma$ act linearly independently on the representation space $\Ind_P^G H_\sigma{\otimes} \C_{i\varphi}$, which therefore decomposes into a direct sum of $|R_\sigma|$ distinct irreducible subrepresentations,
\begin{equation}
    \label{eq-decomposition-of-ind-sigma-tensor-exp-i-nu}
\Ind_P^G H_\sigma{\otimes} \C_{i \varphi} = \bigoplus _\mu X_{\mu, \varphi},
\end{equation}
on each of which $R_\sigma$ acts as multiples of a distinct character.  So we can write 
\begin{equation}
    \label{eq-direct-sum-decomposition-by-r-group}
C_0 \bigl ( \mathfrak{a}_P^{*,R_\sigma} , \Compact ( \Ind_P^G H_\sigma)\bigr )^{R_\sigma} = 
\bigoplus _\mu  C_0( \mathfrak{a}_P^{*,R_\sigma} , \Compact ( \boldsymbol{X_{\mu}})),
\end{equation}
where $\boldsymbol{X_{\mu}}$ is the bundle of Hilbert spaces with fibers $X_{\mu, \varphi}$. We can therefore form the $C^*$-algebra morphism
\begin{multline}
    \label{eq-restriction-projection}
C_0 \bigl ( \mathfrak{a}_P^{*} , \Compact ( \Ind_P^G H_\sigma)\bigr )^{R_\sigma} 
\longrightarrow
C_0 \bigl ( \mathfrak{a}_P^{*,R_\sigma} , \Compact ( \Ind_P^G H_\sigma)\bigr )^{R_\sigma} 
\\
\longrightarrow
C_0( \mathfrak{a}_P^{*,R_\sigma} , \Compact ( \boldsymbol{X_{\mu}}))
\end{multline}
in which the first map is restriction to 
$\mathfrak{a}_P^{*,R_\sigma}\subseteq \mathfrak{a}_P^*$ and the second is projection to a single summand in \eqref{eq-direct-sum-decomposition-by-r-group}. The target $C^*$-algebra is Morita equivalent to $C_0( \mathfrak{a}_P^{*,R_\sigma} )$ via the bimodule $C_0( \mathfrak{a}_P^{*,R_\sigma} , \boldsymbol{X_\mu})$ and so we can formulate a more precise version of Theorem~\ref{thm-k-theory-0f-essential-component} as follows:

\begin{theorem}
\label{thm-k-theory-0f-essential-component2}
If $[P,\sigma]$ is essential, then for every $\mu$ the restriction-projection morphism
\[
C_0\bigl (\mathfrak{a}^*_{P} , \Compact (\Ind_P^G H_\sigma )\bigr)^{R_\sigma}   
\longrightarrow 
C_0\bigl (\mathfrak{a}^{*,R_\sigma}_{P} , \Compact (\boldsymbol{X_\mu} ) \bigr)   
\]
in \eqref{eq-restriction-projection} induces an isomorphism 
\[
K_*\bigl (C_0\bigl (\mathfrak{a}^*_{P} , \Compact (\Ind_P^G H_\sigma  )\bigr)^{R_\sigma}  \bigr)
\stackrel \cong \longrightarrow K_* \bigl (C_0(\mathfrak{a}^{*,R_\sigma}_{P})\bigr ) .
\]
\end{theorem}

\begin{proof}
We shall prove the Morita equivalent version of the theorem that uses $\ell^2 R_\sigma$. The $C^*$-algebra $C_0(\mathfrak{a}_P^*, \mathfrak{K}(\ell^2 R_\sigma))^{R_\sigma}$ admits a tensor product decomposition
 \begin{multline}
 \label{eq-tensor-product-decomposition}
 C_0(\mathfrak{a}_P^*, \mathfrak{K}(\ell^2 R_\sigma))^{R_\sigma}
 \cong  C_0(\R^d) \otimes C_0(\R, \mathfrak{K} (\ell^2 \Z_2))^{\Z_2}\otimes \cdots
\\
 \cdots \otimes C_0(\R, \mathfrak{K} (\ell^2 \Z_2))^{\Z_2},
\end{multline}
where $d$ is the dimension of the subspace of $\mathfrak{a}_P^*$ fixed by the $R_\sigma$-action, and where there are as many factors of $C_0(\R, \mathfrak{K} (\ell^2 \Z_2))^{\Z_2}$ as there are factors of $\Z_2$ in the group $R_\sigma$. 
Now each of the fixed point algebras in the factorization above fits in an extension 
\[
0 \longrightarrow \mathfrak{J} \longrightarrow 
C_0(\R, \mathfrak{K} (\ell^2 \Z_2))^{\Z_2} \stackrel \pi \longrightarrow \C \to 0
\]
where the quotient map $\pi$ is evaluation at $0\in \R$, followed by compression to the subspace of constant functions in $\ell^2\Z_2$. 
The kernel $\mathfrak{J}$ is Morita equivalent to the contractible $C^*$-algebra of $C_0$-functions on $[0,\infty)$.  It follows that the tensor product of the morphisms $\pi$ above gives an isomorphism in $K$-theory from the right-hand side in \eqref{eq-tensor-product-decomposition} to $C_0(\R^d)$, and  since the tensor product of the morphisms $\pi$ is the same as \eqref{eq-restriction-projection}, the theorem follows.
\end{proof}

Let us summarize:

\begin{theorem}
\label{thm-k-theory-c-star-g-summarized}
The group $K_{\dim (G/K)}(C^*_r (G))$ is a free abelian group on the set of essential associate classes,   while the group 
$
K_{\dim (G/K)+1}(\Csr  (G))$ is zero. More precisely, the $K$-theory of each essential summand of $C^*_r(G)$ is free abelian in one generator in degree $\dim(G/K)$, while the $K$-theory of each inessential summand of $C^*_r(G)$ is zero.
\end{theorem}

\section{The Connes-Kasparov Index Homomorphism}
\label{sec-ck-iso}

In this section we shall review the construction of Dirac operators on the symmetric space $K\backslash G$ of right $K$-cosets in $G$ and the definition of the Connes-Kasparov index homomorphism.

We shall fix for the rest of the paper a  $G$-invariant symmetric bilinear form 
\begin{equation}
    \label{eq-bilinear-form-b}
B\colon \mathfrak{g}\times \mathfrak{g} \longrightarrow \R
\end{equation}
that is positive-definite on $\mathfrak{s}$ and negative-definite on $\mathfrak{k}$ (where   $\mathfrak{g}= \mathfrak{k} \oplus\mathfrak{s}$ is the Cartan decomposition that was fixed at the beginning of Section~\ref{sec-parabolic-induction}).

\subsection*{Spin Modules}
The definitions and results in this section may be applied to any orthogonal action of a compact Lie group $K$ on a finite-dimensional Euclidean vector space  $\mathfrak{s}$. But soon the only example of interest will be the adjoint action of the maximal compact subgroup $K\subseteq G$  on the space $\mathfrak{s}$ in the Cartan decompostion  $\mathfrak{g}= \mathfrak{k} \oplus\mathfrak{s}$.  With this in mind  we shall denote the $K$-action   by \[(k,X)\longmapsto \operatorname{Ad}_k(X)\] for $k\in K$ and $X \in \mathfrak{s}$ (even though this is a slight abuse of notation in general).

Form the Clifford algebra $ \Cliff  (\mathfrak{s})$ using the  convention   that the square of any element  from $\mathfrak{s}$ is \emph{minus} the norm-squared of that element, for the inner product $\mathfrak{s}$.\footnote{This convention agrees with \cite{HuangPandzic}, which is one of the main references for the material here, but it disagrees with \cite{Meinrenken13}, which is the other main reference.} If $X \in \mathfrak{s}$, then we shall denote by $c(X)$ the corresponding element in the Clifford algebra, so that the  convention reads: \[c(X)^2=-\|X\|^2\cdot1.\]

\begin{definition} \label{defn:spin module}
If $\mathfrak{s}$ is even-dimensional, then a \emph{spin module} for  the pair $(K,\mathfrak{s})$ is a finite-dimensional, $\Z_2$-graded, complex Hilbert space  $S$ that is equipped with
\begin{enumerate}[\rm (i)]

\item  a representation of $\Cliff(\mathfrak{s})$, written
\[
(X,s) \longmapsto X \cdot s 
\]
for $X\in \mathfrak{s}$ and $s \in S$, in which each  $X$ acts as a grading-degree one, skew-adjoint operator, and 

\item a grading-degree zero, unitary representation of $K$ that is compatible with the representation of $\Cliff (\mathfrak{s})$ in the sense that 
\[
k\cdot (X \cdot s) = \operatorname{Ad}_k(X)\cdot (k\cdot s) 
\]
for every $k\in K$, every $X \in \mathfrak{s}$, and every $s\in S$.

\end{enumerate}
If $\mathfrak{s}$ is odd-dimensional,   then a {spin module} for  $(K,\mathfrak{s})$ is a spin module for $(K,\mathfrak{s}\oplus \R)$, where $\R$ is equipped with the trivial action of $K$. A spin module for $(K, \mathfrak{s})$ is \emph{irreducible} if it cannot be decomposed into a direct sum of two spin submodules. 
\end{definition}
 
 \begin{definition} 
 We shall denote by 
 $R_{\spin}(K,\mathfrak{s})$ the abelian group generated by isomorphism classes of spin modules subject to the relations 
 \[
 [S_1]+ [S_2] = [S_1\oplus S_2]
 \quad \text{and} \quad 
 [S]+ [S^{\text{\rm opp}}] = 0,
 \]
 where $S^{\text{\rm opp}}$ is obtained from $S$ by reversing the $\Z_2$-grading.
 \end{definition}
 
 This group may be analyzed using the following standard construction from Clifford algebras and Lie theory (compare \cite[Sec.~2.3]{HuangPandzic} or \cite[Sec.~2.2.10]{Meinrenken13}): 
 
 \begin{definition}
 \label{def-lie-clifford-morphism}
 The \emph{fundamental morphism}
  \[
 \alpha \colon \mathfrak{k} \longrightarrow \Cliff (\mathfrak{s}) 
 \]
 is defined by the formula 
  \[
 \alpha (Z) = - \frac 14 \sum c(\ad_Z( X_a))\cdot  
 c(X_a),
 \]
 where the sum is over any orthonormal basis  $\{ X_a\}$ of $\mathfrak{s}$ (the sum is independent of the choice).
 \end{definition}

The fundamental morphism is a Lie algebra morphism (for the commutator bracket on the Clifford algebra) and moreover 
\[
c(\ad_Z(X)) = [\alpha(Z),c(X)]
\]
for all $X\in \mathfrak{s}$ and all $Z\in\mathfrak{k}$.  If $S$ is a spin module, then by composing the Clifford algebra action on $S$ with $\alpha$ we obtain a representation of  $\mathfrak{k}$ on $S$.  

Suppose for a moment that $\mathfrak{s}$ is even-dimensional. Fix  an irreducible representation $S_{\text{irr}}$   of the Clifford algebra on a finite-dimensional $\Z_2$-graded Hilbert space ($S_{\text{irr}}$ is unique up to a possibly grading-reversing unitary equivalence). The fundamental morphism endows $S_{\text{irr}}$ with a $\mathfrak{k}$-action, and so to any spin module $S$ we can associate the $\Z_2$-graded  $\mathfrak{k}$-module
\begin{equation}
    \label{eq-def-of-mod-s}
\module(S) = \operatorname{Hom}_{\Cliff (\mathfrak{s})} (S, S_{\text{\rm irr}}),
\end{equation}
(the morphisms in $\module(S) $ need not be grading-preserving). We can reconstruct $S$ from $\module(S)$ via the canonical isomorphism
\begin{equation}
    \label{eq-canonical-morphism-for-s}
 S_{\text{\rm irr}}\widehat{\otimes} \module (S)^*\stackrel \cong \longrightarrow S  ,
\end{equation}
where the tensor product is given the diagonal $\mathfrak{k}$-action.  Note that as a result, if $S$ is irreducible in the sense of Definition~\ref{defn:spin module}, then $\module(S)$ is an irreducible $\mathfrak{k}$-module.

If $\mathfrak{s}$ is odd-dimensional, then  we repeat the above with $\Cliff(\mathfrak{s}\oplus \R)$ in place of $\Cliff(\mathfrak{s})$.
In either case, the $\mathfrak{k}$-module $\module(S)$ does not necessarily integrate to a representation of $K$. However if we define a compact group $\tilde K$ by means of the  pullback    diagram
\[
\xymatrix{
\tilde K \ar[r]\ar[d] &  \operatorname{Spin} (\mathfrak{s})\ar[d]  \\
K\ar[r] & SO(\mathfrak{s})  \\
}
\]
in which the bottom morphism comes from the adjoint action of $K$ on $\mathfrak{s}$, then $\module(S)$ integrates to a representation of $\tilde K$. 

The pullback group  $\tilde K$  may or may not be connected, but in any case the kernel of the morphism   $\tilde K{\to}K$ is a two-element group, and there is a unique morphism from $\tilde K$ into the group of invertible elements in $\Cliff (\mathfrak{s})$ whose associated Lie algebra morphism is $\alpha$, and which maps the nontrivial element   of the kernel to minus the identity.

\begin{definition}
We shall say that a representation of $\tilde K$ is \emph{genuine} if the nontrivial element in the kernel of $\tilde K {\to} K$ acts as $-I$ in the representation.
\end{definition}

\begin{theorem}[{See for instance \cite[Thm~0.1]{EchterhoffPfante}}]
The abelian group $R_{\spin}(K,\mathfrak{s})$ is isomorphic via the correspondence $S \mapsto \module (S)$  to the free abelian group on the set of equivalence classes of irreducible and  genuine representations of $\tilde K$.  \qed
\end{theorem}

\subsection*{The Dirac Operator and its Square} 

For the rest of this section $K$  will be the given maximal compact subgroup of our real reductive group $G$, and $\mathfrak{s}$ will be the complementary subspace to $\mathfrak{k}$ in the Cartan decomposition $\mathfrak{g}=\mathfrak{k}\oplus \mathfrak{s}$.  We shall equip $\mathfrak{s}$ with the inner product  coming from   \eqref{eq-bilinear-form-b}.

Given a spin module $S$,   form the space $[C_c^\infty (G)\otimes S]^K$, where $K$ acts diagonally, and where the $K$-action on $C_c^\infty (G)$ is by left translation.

  \begin{definition}
\label{def-dirac-operator}
The \emph{Dirac operator} associated to a spin module $S$ is the  linear operator
   \[ 
 \slashed{D}_S \colon [C_c^\infty(G)\otimes   S  ] ^K \longrightarrow  [C_c^\infty(G)\otimes S ] ^K 
 \]
 given by the formula 
 \[
 \slashed{D}_S = \sum X_a \otimes c(X_a),
 \]
 in which  the sum is over any orthonormal basis $\{ X_a\}$ for $\mathfrak{s}$, as in Definition \ref{def-lie-clifford-morphism}, and $X_a$ acts on $C_c^\infty(G)$ via the left-translation action of $G$ on $C_c^\infty (G)$.
Compare  \cite{Parthasarathy,AtiyahSchmid}.
 \end{definition}

\begin{definition} 
The \emph{Casimir element} for $G$ is the element
\[
 \Omega_G =\sum Y^a Y_a 
\]
in the enveloping algebra\footnote{Generally we shall work with the complexification of the enveloping algebra of the real Lie algebra $\mathfrak{g}$, or equivalently the enveloping algebra of the complexification, but here it is not necessary to do so.}
$\mathcal{U}(\mathfrak{g})$, where the sum is over any basis $\{Y_a\}$ for $\mathfrak{g}$ with dual basis $\{ Y^a\}$ for the invariant form $B$ (so that $B(Y^a,Y_b) = \delta^a_b$).  Similarly the {Casimir element} for $K$ is the element
\[
 \Omega_K =\sum Z^b Z_b \in \mathcal{U}(\mathfrak{k}) ,
\]
where the sum is over any basis $\{Z_b\}$ for $\mathfrak{k}$ and   dual basis $\{ Z^b\}$ for the invariant form $B$, restricted to $K$.
\end{definition}

\begin{definition} 
The 
\emph{diagonal morphism}
\[
\Delta\colon \mathcal{U}(\mathfrak{k}) \longrightarrow \mathcal{U}(\mathfrak{g})\otimes \Cliff (\mathfrak{s})
\]
is the morphism of associative algebras for which
\[
\Delta(Z) =  Z{\otimes} I + I {\otimes} \alpha (Z)
\]
for all  $Z\in \mathfrak{k}$, where $\alpha$ is the fundamental morphism from   Definition~\ref{def-lie-clifford-morphism}.
\end{definition}

In the  next result, it is convenient to view  the Dirac operator algebraically, as an element of $\mathcal{U}(\mathfrak{g})\otimes \Cliff (\mathfrak{s})$; the choice of spin module $S$ is therefore no longer immediately relevant.  The expression for the square of the Dirac operator in Theorem \ref{thm-parthasarathy-formula} below is essentially due to Parthasarathy \cite{Parthasarathy}; see \cite[Prop.~3.1.6]{HuangPandzic} for a modern account. 

The bilinear form $B$ in \eqref{eq-bilinear-form-b} may be extended  to a nondegenerate symmetric complex-bilinear form 
\begin{equation}
    \label{eq-complex-bilinear-form-B}
B\colon \mathfrak{g}_{\C} \times \mathfrak{g}_{\C} \longrightarrow \C
\end{equation}
on the complexification of $\mathfrak{g}$.  This restricts to a nondegenerate bilinear form on each Cartan subalgebra of $\mathfrak{g}_{\C}$, and also on each Cartan subalgebra of $\mathfrak{k}_{\C}$. These restrictions may be used to identify the Cartan subalgebras with their complex vector space duals in the usual way,  and so we obtain nondegenerate complex-bilinear forms $B^*$ on these dual spaces.

\begin{theorem} 
\label{thm-parthasarathy-formula}
Let $\rho_K$ and $\rho_G$ be the half-sums of the positive roots  for $\mathfrak{k}_{\C}$ and $\mathfrak{g}_{\C}$  \textup{(}formed using any choices of Cartan subalgebras in the complexified Lie algebras and any systems of positive roots\textup{)}. 
The square of the Dirac operator
in $\mathcal{U}(\mathfrak{g})\otimes \Cliff (\mathfrak{s})$
is  given by the formula
\[
\slashed{D}^2 =        \Delta  \bigl (\Omega_K + B^*(\rho_K,\rho_K)     \bigr )    -   \bigl (\Omega_G   + B^*(\rho_G,\rho_G) \bigr ).
\]
\end{theorem}

 The scalars $B^*(\rho_K,\rho_K)$ and $B^*(\rho_G,\rho_G) $ are independent of  the choices of  Cartan subalgebras and systems of positive roots. Moreover they are  real-valued, and indeed non-negative. This may be seen by selecting a (complexification of a) $\theta$-stable Cartan subalgebra $\mathfrak{h} = \mathfrak{t} \oplus \mathfrak{a}\subseteq \mathfrak{g}$ with $\mathfrak{t}$ maximal abelian in $\mathfrak{k}$ and $\mathfrak{a}$ orthogonal to $\mathfrak{k}$, and observing that $\rho_K$ and $\rho_G$ are imaginary-valued on $\mathfrak{t}$, where $B$ is positive definite, while $\rho_G$ is real-valued on $\mathfrak{a}$, where $B$ is real-valued. 

In order to  reflect    this non-negativity, it will be convenient to change the notation in Theorem~\ref{thm-parthasarathy-formula}.  The formula 
\begin{equation}
    \label{eq-def-of-inner-product}
\langle X,Y\rangle = - B(X , \theta (Y))  \qquad (X,Y\in \mathfrak{g}).
\end{equation}
defines a positive-definite inner product on the real Lie algebra $\mathfrak{g}$. Let us extend it to a complex-sesquilinear  inner product on the complexification $\mathfrak{g}_{\C}$. Of course, this extension  restricts to an inner product on any Cartan subalgebra of the complexification, and   the restriction induces an inner product on the vector space dual of the Cartan subalgebra.  The same goes for any Cartan subalgebra of the Lie algebra of $K$. With this notation, Theorem~\ref{thm-parthasarathy-formula} may be restated as 
\begin{equation}
    \label{eq-parthasarathy-formula-with-norms}
\slashed{D}^2 =        \Delta  \bigl (\Omega_K + \|\rho_K\|^2     \bigr )    -   \bigl (\Omega_G   + \|\rho_G\|^2 \bigr ).
\end{equation}

Let us now bring the spin module $S$ back into the picture and compute the operator 
\[
\Delta   \bigl (\Omega_K + \|\rho_K\|^2     \bigr )
\colon  [C_c^\infty (G)\otimes S]^K \longrightarrow 
[C_c^\infty (G)\otimes S]^K 
\]
arising from  Theorem~\ref{thm-parthasarathy-formula}. 

 \begin{definition}
 \label{def-norm-of-spin-module}
 If $S$ is an irreducible spin module (Definition \ref{defn:spin module}), 
 then we define 
 \[
 \|S\| = 
 \| \mu + \rho _K\| ,
 \]
where $\mu$ is the  highest weight of the irreducible $\mathfrak{k}$-module $\module (S)$  (both $\mu$ and $\rho_K$ depend on a choice of Cartan subalgebra of $\mathfrak{k}_{\C}$  and system of positive roots, but the norm does not).  
 \end{definition}
 
\begin{lemma}
\label{lem-formula-for-delta-term}
If $S$ is an irreducible spin module, then the operator
\[
\Delta \bigl (\Omega_K + \|\rho_K\|^2 \bigr )   \colon [C_c^\infty (G) \otimes S]^K\longrightarrow [C_c^\infty (G)\otimes S]^K
\]
is $ \|S\|^2$ times the identity operator.
\end{lemma}

\begin{proof}
Write $S \cong S_{\text{irr}}\otimes \module(S)^*$ as in \eqref{eq-canonical-morphism-for-s}. Under this isomorphism, the Clifford algebra acts on  $S_{\text{irr}}$, but not on $\module(S)^*$.  So the diagonal morphism $\Delta$ gives the action 
\[
\Delta(Z) = Z\otimes 1 \otimes 1 + 1 \otimes \alpha (Z) \otimes 1
\]
of the Lie algebra $\mathfrak{k}$ on  
\[
C_c^\infty (G) \otimes S\cong C_c^\infty (G) \otimes S_{\text{irr}}\otimes \module(S)^*.
\]
In contrast,  the $K$-fixed part of this space is computed using the actions  of $K$ (or, strictly speaking $\tilde K$) on all three factors, so that 
\[
\bigl [C_c^\infty (G) \otimes S \bigr ] ^K \cong \operatorname{Hom}_K \bigl ( \C, C_c^\infty (G)\otimes S\bigr ) \cong \operatorname{Hom}_{\tilde K} \bigl (\module(S), C_c^\infty (G) \otimes S_{\text{\rm irr}}\bigr ) .
\]
We can therefore compute the action of  $\Delta (\Omega_K)$ using either the action of $\mathfrak{k}$ on $C_c^\infty (G) \otimes S_{\text{\rm irr}} $ or the action of $\mathfrak{k}$ on $\module(S)$. Using the  latter it is well known that we obtain  $\|\mu + \rho_K\|^2-\|\rho_K\|^2$; See \cite[Prop 4.120]{KnappVoganBook} for instance.
\end{proof}

With this, we can simplify the formula for the square of the Dirac operator: 

\begin{theorem} 
\label{thm-parthasarathy-formula2}
The square of the Dirac operator
\[
\slashed{D}_S \colon [C_c^\infty (G)\otimes S]^K \longrightarrow 
[C_c^\infty (G)\otimes S]^K 
\]
associated to an irreducible spin module is given by the formula
\[
\pushQED{\qed} 
\slashed{D}_S^2 =  \|S\|^2  -   \bigl (\Omega_G   + \|\rho_G\|^2\bigr ).
\qedhere
\popQED
\]
\end{theorem}

\subsection*{Infinitesimal Characters}

Let us quickly review some basic topics in representation theory.  For a further discussion of all these  concepts and results, see for instance \cite[Ch.~VIII]{Knapp1}.

If $\pi$ is any  continuous representation of $G$ on a Hilbert space $H_\pi$ by bounded, invertible operators, then we shall denote by $H_{\pi,\fin}$ the space of $K$-finite vectors in $H_\pi$.  The representation $\pi$ is said to be  \emph{admissible} if  each $K$-isotypical subspace in $H_{\pi,\fin}$ is finite-dimensional. According to a theorem of Harish-Chandra, if $\pi$ is unitary and irreducible, then it is admissible.  

If $\pi$ is admissible (but not necessarily unitary), then $H_{\pi,\fin}$ is included within the smooth vectors in $H_\pi$   and so it carries a representation of the complexified Lie algebra $\mathfrak{g}_{\C}$. One says that $\pi$ is  \emph{quasi-simple} if  $\mathcal{Z}(\mathfrak{g}_{\C})$, the center of the universal enveloping algebra of $\mathfrak{g}_{\C}$, acts as multiples of the identity operator.

Let $\mathfrak{h}$ be any Cartan subalgebra of $\mathfrak{g}$. Harish-Chandra defined an isomorphism
\begin{equation}
    \label{eq-harish-chandra-homo}
HC\colon \mathcal{Z}(\mathfrak{g}_{\C})\stackrel \cong \longrightarrow \mathcal{S}(\mathfrak{h}_{\C})^W ,
\end{equation}
where  $W$ is the Weyl group associated to $(\mathfrak{g},\mathfrak{h})$ and  $\mathcal{S}(\mathfrak{h}_{\C})^W $ is the $W$-invar\-iant part of the symmetric algebra of $\mathfrak{h}_{\C}$, the complexification of the real Lie algebra $\mathfrak{h}$.    We shall identify the range in \eqref{eq-harish-chandra-homo} with the algebra of $W$-invariant complex polynomial functions on the complex vector space $\operatorname{Hom}_{\R} (\mathfrak{h},\C)$.

If $\pi$ is an admissible and  quasi-simple   representation of $G$, then the \emph{infinitesimal character} of $\pi$ is the algebra homomorphism 
\[
\infch(\pi) \colon \mathcal{Z}(\mathfrak{g}_{\C}) \longrightarrow \C 
\]
that gives the action of the center of the enveloping algebra on $H_{\pi,\fin}$.
Using the Harish-Chandra isomorphism we can and will view the infinitesimal character as (any representative of) a $W$-orbit in $\operatorname{Hom}_{\R} (\mathfrak{h},\C)$.

\begin{definition}\label{def-real_im_Cartan}
Let $\mathfrak{h} = \mathfrak{t}_{\mathfrak{h}} \oplus \mathfrak{a}_{\mathfrak{h}} $ be a $\theta$-stable Cartan subalgebra of $\mathfrak{g}$, with $\mathfrak{t}_{\mathfrak{h}}\subseteq \mathfrak {k}$ and $\mathfrak{a}_{\mathfrak{h}}\subseteq \mathfrak {s}$.   An element of $\operatorname{Hom}_{\R} ( \mathfrak{h}, \C)$ is said to be \emph{real} if it belongs to 
\[
\operatorname{Hom}_\R (\mathfrak{t}_{\mathfrak{h}}, i \R)
\oplus
\operatorname{Hom}_\R (\mathfrak{a}_{\mathfrak{h}},   \R)\]
and it is said to be \emph{imaginary} if it belongs to 
\[
\operatorname{Hom}_\R (\mathfrak{t}_{\mathfrak{h}},   \R)
\oplus
\operatorname{Hom}_\R (\mathfrak{a}_{\mathfrak{h}},   i \R) .
\]
\end{definition}

The same  terminology may be applied to infinitesimal characters, using the Harish-Chandra isomorphism. Whether or not an infinitesimal character is real or imaginary  does not depend on the choice of representative of $\infch(\pi)$ within  its $W$-orbit.  Nor does it depend on the choice of Cartan subalgebra (as long as the Cartan subalgebra is stable under the Cartan involution). See \cite[p. 535]{Knapp1}.

For the next result, recall  that the complexification process outlined in the discussion  preceding  \eqref{eq-parthasarathy-formula-with-norms} 
endows $ \operatorname{Hom}_{\R} ( \mathfrak{h}, \C)$ with an inner product. Every element of $\operatorname{Hom}_{\R} ( \mathfrak{h}, \C)$ decomposes as a sum of real and imaginary parts, and we shall use the standard notation for these.

\begin{lemma}
\label{lem-action-of-casimir-and-inf-ch}
Let $\pi$ be a    unitary, admissible and quasi-simple  representation of $G$ on a Hilbert space $H_\pi$.  If the real and imaginary parts of the infinitesimal character of $\pi$ are orthogonal, then the operator $ \pi (\Omega_G)+ \|\rho_G\|^2 $ acts on $H_{\pi,\fin}$ as the scalar 
\[
\| \Re (\infch(\pi))\|^2  -  \| \Im (\infch(\pi))\| ^2  .
\]
\end{lemma}

\begin{proof}
The formula in the statement of the lemma is a special case of  the following standard identity for the Harish-Chandra homomorphism \eqref{eq-harish-chandra-homo}: 
 \begin{equation*}
  HC (\Omega_G) (\lambda  )+ B^* (\rho_G, \rho_G) = B^*  (\lambda, \lambda  )\qquad \forall \lambda \in \mathfrak{h}_{\mathbb{C}}^*
 \end{equation*}
(we introduced the form $B^*$  in the discussion prior to the statement of Theorem~\ref{thm-parthasarathy-formula}).   For a proof see for instance 
\cite[Prop.~4.120]{KnappVoganBook}; the relation between the form $\left\langle\cdot,\cdot\right\rangle$ that appears there and  our $B^*$ is explained on
 p.299 of \cite{KnappVoganBook}.
\end{proof}

Let us now apply this to the  unitary principal series representations of $G$.
Let $[P,\sigma]$ be an associate class, 
and let  $P = M_PA_PN_P$ be the Langlands decomposition of $P$, so that $\sigma$ is an irreducible square-integrable representation of $M_P$.    Harish-Chandra showed that: 

\begin{theorem}  Whenever $M_P$ carries an irreducible square-integrable representation, the Lie algebra $\mathfrak{t}_P$  of any maximal torus in $K\cap M_P$ is a Cartan subalgebra of $\mathfrak{m}_P$.  
\end{theorem}

\begin{theorem}
\label{thm-harish-chandra-weak-classification}
Every irreducible, square-integrable representation of $M$ has real infinitesimal character.  Moreover, for every $N>0$ the set of equivalence classes of irreducible, square-integrable representations of $M$ with $
\| \infch (\sigma) \| < N 
$
is finite.
\end{theorem}

For an exposition of these results, see for instance \cite{Knapp1}. We shall use the second statement in the second theorem in the next subsection. As a result of the first theorem, the Lie algebra $\mathfrak{t}_P \oplus \mathfrak{a}_P$ is a Cartan subalgebra of $\mathfrak{g}$, and we may compute the infinitesimal characters for the $(P,\sigma)$-principal series as follows:

\begin{lemma}[{See for example \cite[Prop.~8.22]{Knapp1}}]
\label{lem-inf-ch-princ-series}
The infinitesimal character of the unitary  $(P,\sigma)$-principal series representation $\pi_{\sigma, \varphi}$ is 
\[
\infch(\sigma) \oplus \varphi \in 
\operatorname{Hom}_\R (\mathfrak{t}_P, i \R)
\oplus
\operatorname{Hom}_\R (\mathfrak{a}_P, i \R).
\]
\end{lemma} 

Note that the two summands in the infinitesimal character  above are its real and imaginary parts, respectively.

\subsection*{Dirac Operator from the Representation Theory Point of View}

Now form   the Hilbert space $\Ind_P^G H_\sigma $ as in Definition~\ref{def-induced-representation-hilbert-space}, and given a    spin-module $S$, form the fixed space 
\[
[\Ind_P^G H_\sigma  \otimes S]^K .
\]
The same space is obtained if we replace $\Ind_P^G H_\sigma$ by its subspace of $K$-finite vectors, and as a result   $[\Ind_P^G H_\sigma\otimes S]^K$ carries an action of     $\mathfrak{g}$.  So if we regard $\Ind_P^G H_\sigma$ as carrying the principal series representation $\pi_{\sigma, \varphi}$, then  we may form the operator
\begin{equation}\label{eq-Dirac operator}
\slashed{D}_{\sigma,\varphi,S}  = \sum \pi_{\sigma, \varphi}(X_a) \otimes c(X_a)\colon [\Ind_P^G H_\sigma \otimes S] ^K
\longrightarrow  [ \Ind_P^G H_\sigma \otimes S] ^K .
\end{equation}
The following is an immediate consequence of Lemmas~\ref{lem-action-of-casimir-and-inf-ch}  and \ref{lem-inf-ch-princ-series}:

\begin{lemma}
The operator $ \pi_{\sigma, \varphi} (\Omega_G)+ \|\rho_G\|^2$ acts on $[\Ind_P^G H_{\sigma} \otimes S]^K $ as the scalar 
$ \|  \infch(\sigma)\|^2  -  \| \varphi \| ^2$.
 \qed
\end{lemma}

Putting this together with Theorems~\ref{thm-parthasarathy-formula} and \ref{thm-parthasarathy-formula2}, we arrive at the following result (compare \cite[p.562]{NoteWassermann}): 

\begin{theorem}
\label{thm-square-of-principal-series-dirac-ops}
If $\pi_{\sigma,\varphi}$ is any $(P,\sigma)$-principal series representation, and if $S$ is any irreducible spin module, then 
\[
\pushQED{\qed} 
\slashed{D}_{\sigma,\varphi,S}^2 =    \|S\|^2 - \|\infch (\sigma) \|^2  + \| \varphi\|^2  .
\qedhere\popQED
\]
\end{theorem}

\begin{remark} Strictly speaking, to reach this conclusion we need the formula 
\[
\Delta \bigl (\Omega_K + \|\rho_K\|^2 \bigr ) = \| S\|^2 \colon 
 [\Ind_P^G H_\sigma \otimes S] ^K
\longrightarrow  [ \Ind_P^G H_\sigma \otimes S] ^K ,
\]
 which is a version of Lemma~\ref{lem-formula-for-delta-term} with $C_c^\infty (G)$ replaced by the $K$-finite vectors in $\Ind_P^G H_\sigma$.  This follows by a verbatim repetition of the proof of Lemma~\ref{lem-formula-for-delta-term}.  
\end{remark}

\begin{corollary}
\label{cor-uniform-admissibility-and-dirac-operator}
For every spin module  $S$, the space  $[\Ind_P^G H_\sigma \otimes S]^K$ is zero for all but finitely many associate classes $[P,\sigma]$.
\end{corollary}
 
\begin{proof} 
The formula 
\[
\slashed{D}_{\sigma,0,S}^2 =    \|S\|^2 - \|\infch (\sigma) \|^2    
\]
shows that $\slashed{D}_{\sigma,0,S}^2$ will be negative whenever $\|\infch (\sigma) \|^2> \|S\|^2$, assuming that  $[\Ind_P^G H_\sigma \otimes S]^K$ is nonzero. But the Dirac operator is self-adjoint, so its square   is  positive semidefinite.  So necessarily  $[\Ind_P^G H_\sigma \otimes S]^K$ is zero in these cases.  The corollary now follows from  
Theorem~\ref{thm-harish-chandra-weak-classification}.
\end{proof}

\subsection*{Dirac Operator from a Functional Analytic Point of View }

The Dirac operator $\slashed{D}_S$ can be viewed as an unbounded operator on the Hilbert space  $[L^2(G)\otimes S] ^K$ with domain  $[C_c^\infty(G)\otimes   S  ] ^K$. The Dirac operator so viewed is essentially self-adjoint (see \cite{Chernoff73}), and there is therefore an associated one-parameter  group of unitary operators $\exp( i t \slashed{D}_S)$.  These restrict to operators 
\begin{equation}
    \label{eq-finite-propagation-unitaries}
\exp( i t \slashed{D}_S) \colon [C_c^\infty(G)\otimes   S  ] ^K
\longrightarrow
[C_c^\infty(G)\otimes   S  ] ^K
\end{equation}
(see \cite{Chernoff73} again); this is the finite propagation property of the Dirac operator.  
 
But we are more interested in viewing $\slashed{D}_S$ as an unbounded operator on the space  $[\Csr  (G)\otimes S ] ^K$, which becomes a   \emph{Hilbert $C^*$-module} over $C^*_r (G)$ when equipped with the right action of $C^*_r (G)$ on the first factor and the $C^*_r (G)$-valued inner product
  \[
 \langle f _1\otimes s_1  , f_2\otimes s_2  \rangle = f_1^*f_2 
  \langle   s_1, s_2  \rangle .
 \]
See \cite{Lance} for general information about Hilbert $C^*$-modules. 
The   operators $\exp( i t \slashed{D}_S)$ in \eqref{eq-finite-propagation-unitaries} extend to a one-parameter   group of unitary operators 
\begin{equation*}
    \label{eq-finite-propagation-unitaries2}
\exp( i t \slashed{D}_S) \colon [C^*_r(G)\otimes   S  ] ^K
\longrightarrow
[C^*_r(G)\otimes   S  ] ^K ,
\end{equation*}  
and the generator of this one-parameter group is a regular and self-adjoint operator on the Hilbert module $[C^*_r(G)\otimes   S  ] ^K$  in the sense of \cite[Ch.~9]{Lance}.  We shall use the same notation $\slashed{D}_S$ for the extension.

\begin{definition}
The \emph{bounded transform} of $\slashed{D}_S$ is the operator 
 \[
 \slashed{F}_S = \slashed{D}_S (I + \slashed{D}_S^2)^{-1/2} \colon  [\Csr  (G)\otimes S ] ^K \longrightarrow  [\Csr  (G)\otimes S ] ^K 
 \]
that is defined using the functional calculus for regular self-adjoint Hilbert module operators.
 \end{definition}

As in the previous section, given a  $(P,\sigma)$-principal series representation 
\[
\pi_{\sigma, \varphi}\colon 
G  \longrightarrow  U\bigl ( \Ind_P^G H_\sigma\bigr ) 
\]
we may form the operator
\[
\slashed{D}_{\sigma,\varphi,S}  = \sum \pi_{\sigma, \varphi}(X_a) \otimes c(X_a)\colon [\Ind_P^G H_\sigma \otimes S] ^K
\longrightarrow  [ \Ind_P^G H_\sigma \otimes S] ^K ,
\]
with $\{ X_a\}$ an orthonormal basis for $\mathfrak{s}$, as usual, and then its bounded transform\footnote{Of course, the operator $\slashed{D}_{\sigma,\varphi,S}$ is acting on a finite-dimensional Hilbert space, and  is therefore already bounded itself. }
\[
\slashed{F} _{\sigma, \varphi, S} = \slashed{D}_{\sigma,\varphi,S} ( I + \slashed{D}_{\sigma,\varphi,S} ^2 ) ^{-1/2}
\colon [\Ind_P^G H_\sigma \otimes S] ^K
\longrightarrow  [ \Ind_P^G H_\sigma \otimes S] ^K .
\]

\begin{lemma} 
\label{lem-formula-for-f-sigma-nu}
Under isomorphism of Hilbert modules \[
 \bigl [ C^*_r (G) \otimes S \bigr ]^K \cong \bigoplus _{[P,\sigma]} \bigl [C_0\bigl (\mathfrak{a}_P^*, \mathfrak{K}(\Ind_P^G H_\sigma)\bigr )^{W_\sigma}
 \otimes S \bigr ]^K 
 \]
associated with the $C^*$-algebra   isomorphism of \textup{Theorem~\ref{thm-CsrG}}, the operator 
\[
\slashed{F}_S \colon \bigl [ C^*_r (G) \otimes S \bigr ]^K  
\longrightarrow \bigl [ C^*_r (G) \otimes S \bigr ]^K 
\]
acts as 
\[
f\otimes s  \longmapsto \Bigl [ \varphi \mapsto \slashed{F}_{\sigma ,\varphi,S} \cdot  ( f(\varphi)\otimes s) \Bigl ]
\]
for all  $f\in C_0  (\mathfrak{a}_P^*, \mathfrak{K}(\Ind_P^G H_\sigma)  )^{W_\sigma} $  and all $s\in S$, where the product $\cdot$ on the right hand side is composition of linear  operators on the finite-dimensional space $[\Ind _P^G H_\sigma \otimes S]^K$. 
\end{lemma}

\begin{proof}
The analogous result for $\exp (it \slashed{D}_S)$ is readily verified on $[C_c^\infty (G)\otimes S]^K$, and the stated result follows from this.
\end{proof}

See \cite[Ch.~1]{Lance} for the meaning of \emph{compact} in  following fundamental result:

\begin{theorem} 
The operator 
\[
I-\slashed{F}_S^2  = (I + \slashed{D}_S^{2})^{-1}
\]
is a compact operator
on the Hilbert module $[C^*_r (G)\otimes S]^K$.
\end{theorem}

 We shall prove this using the representation theory calculations from the  previous section, since those results are at hand.  The original proof, due to \cite{Kasparov83}, uses the basic elliptic estimates for the Dirac operator and the Rellich lemma. See \cite{BCH} for a general account of these matters.

\begin{proof} 
The formula for $\slashed{F}_S$ in Lemma~\ref{lem-formula-for-f-sigma-nu} and the formula for $\slashed{D}_{S,\sigma,\varphi}$ in Theorem~\ref{thm-square-of-principal-series-dirac-ops} combine to give a formula for $1{-}\slashed{F}_S^2$ as an operator on 
\[
 \bigl [ C^*_r (G) \otimes S \bigr ]^K \cong \bigoplus _{[P,\sigma]} \bigl [C_0\bigl (\mathfrak{a}_P^*, \mathfrak{K}(\Ind_P^G H_\sigma)\bigr )^{W_\sigma}
 \otimes S \bigr ]^K 
 \]
 The direct sum here is actually a finite direct sum, in view of Corollary~\ref{cor-uniform-admissibility-and-dirac-operator}, and in each summand $1{-}\slashed{F}_S^2$ acts as multiplication by a $C_0$-scalar-valued function.  Each such operator is compact, thanks to the finite-dimensionality of the spaces $[\Ind_P^G H_\sigma\otimes S]^K$.
\end{proof}

Now the compact operators on any Hilbert $C^*$-module form an ideal in the $C^*$-algebra of all bounded, adjointable operators, and by definition a bounded adjointable operator is \emph{Fredholm} if it is invertible modulo this ideal.  In the present case, we see from the theorem above that $\slashed{F}_S$ is its own inverse, modulo compact operators.  Therefore: 

\begin{corollary} 
The operator 
\[
\slashed{F}_S \colon \bigl [ C^*_r (G)\otimes S\bigr ] ^K \longrightarrow \bigl [ C^*_r (G)\otimes S\bigr ] ^K
\]
is a bounded, self-adjoint, odd-graded, Fredholm operator on the $\Z_2$-graded Hilbert $C^*_r (G)$-module $[C^*_r (G)\otimes S]^K$. \qed
\end{corollary}

\subsection*{The Connes-Kasparov Index Homomorphism}

In order to define the index homomorphism it is convenient to use Kasparov's approach of $C^*$-algebra $K$-theory \cite{Kasparov_Kfunctor} (see \cite[Sec. 3]{KK_primer} for an exposition). 

Kasparov  \emph{defines}  the $K_0$-group of a $C^*$-algebra $A$ as the group of homotopy classes of bounded, self-adjoint, odd-graded, Fredholm operators $F$ on $\Z_2$-graded Hilbert $A$-modules.
In addition, he defines the $K_1$-group in the same way, except that the Hilbert $A$-module $\mathcal{E}$ on which $F$ acts is required to  carry an additional odd-graded skew-symmetry 
\begin{equation}
    \label{eq-odd-grading-symmetry}
\begin{gathered}
\gamma\colon \mathcal{E} \longrightarrow \mathcal{E},\\
  \gamma^* = - \gamma,\quad \gamma^2 =-1
\end{gathered}
\end{equation}
that  anti-commutes with $F$. 

\begin{remark}
We note for later use that, as a consequence of the way homotopy is defined, it is an elementary property of $K$-theory that if a Fredholm operator is actually invertible (not merely invertible modulo compact operators), then it determines the zero class in $K$-theory.
\end{remark}

Kasparov's definitions are made with Dirac operators in mind, and it follows immediately from the definitions and the results we have summarized above that  if $S$ is any spin module, then the Fredholm operator $\slashed{F}_S$ determines a class
 \[
 \Index ( \slashed{F}_S) \in  K_{\dim(G/K)}(\Csr  (G)) 
 \]
 (in the case where $\dim(G/K)$ is odd, the skew-symmetry $\gamma$ is Clifford multiplication by the generator in $\Cliff (\mathfrak{s}\oplus \R)$ associated to the $\R$-summand).
 
 \begin{definition} The \emph{Connes-Kasparov index homomorphism} is the homomorphism of abelian groups 
 \[
 R_{\spin}(K,\mathfrak{s}) \longrightarrow K_{\dim(G/K)} (C^*_r (G))
 \]
that maps the class of a spin module $S$ to the index of  $\slashed{F}_S$ in $K$-theory.
 \end{definition}

Our aim is to prove that: 

\begin{theorem}[Connes-Kasparov Isomorphism]\label{thm-connes-kasparov-isomorphism} If $G$ is a connected, linear real reductive Lie group, then  Connes-Kasparov index map 
\[
  R_{\mathrm{spin}}(K,\mathfrak{s}) \longrightarrow K_{\dim(G/K)}(\Csr  (G))
\]
is an isomorphism of abelian groups. Moreover $K_{\dim(G/K)+1}(\Csr  (G)) =0$.
\end{theorem}

 \begin{remarks}   
\label{rem-on-generalizations-of-connes-kasparov}
There is a Connes-Kasparov index homomorphism for any almost-con\-nected Lie group (that is, any Lie group with finitely many connected components) and moreover it is an isomorphism in this generality \cite{ChabertEchterhoffNest}.  The definition of the index homomorphism for connected Lie groups is essentially the same as the one we have presented.
But beyond connected groups, and  even within the realm of real reductive groups, the definition of the index homomorphism needs to be adjusted \cite{EchterhoffPfante}.   Among other things it is possible   that both $K$-theory groups for $C^*_r (G)$ might be nonzero at the same time, as is the case for $GL(2,\R)$, for instance.   
\end{remarks}

\section{The Matching Theorem}

In this section we shall state a purely representation-theoretic result that will lead quickly (in the next section) to a proof that the Connes-Kasparov index homomorphism is an isomorphism.

\subsection*{Statement of the Matching Theorem}

 \begin{definition}
 \label{def-matching}
 We shall say that an irreducible spin module $S$ for $(K,\mathfrak{s})$ and an associate class  $[P,\sigma]$ are \emph{matched} if 
 \begin{enumerate}[\rm (i)]
 
 \item the space 
 $[ \Ind_P^G H_\sigma \otimes S ]^K$ is nonzero, and 
 
  \item
the Dirac operator $\slashed{D}_{\sigma,0,S}$ vanishes on      $[ \Ind_P^G H_\sigma \otimes S ]^K$.
\end{enumerate}
 \end{definition}

\begin{remark}
\label{rem-norm-of-mod-s}
It follows from Theorem~\ref{thm-square-of-principal-series-dirac-ops} and the fact that the Dirac operator is self-adjoint that the second condition in the definition above is equivalent to the identity  $\|S\| = \|\infch (\sigma)\|$.
 \end{remark}
 
The result that we shall use to establish the Connes-Kasparov isomorphism is as follows:

\begin{theorem}[Matching Theorem]\label{thm-matching}
Let $G$ be a connected linear real reductive group. 
\begin{enumerate}[\rm (i)]

\item For every essential associate class $[P,\sigma]$ there is a unique irreducible spin module $S$ to which $[P,\sigma]$ is matched.
\item For every irreducible spin module $S$ there is a unique essential associate class $[P,\sigma]$  to which $S$ is matched.
\end{enumerate}
\end{theorem}

We shall prove this in a separate article \cite{MatchingTheorem} using a number of important (and quite difficult) results of Vogan from 
\cite{Voganbook}.  But let us give some examples.

\begin{example}
\label{ex-discrete-series-and-h-c-parameters}
If $\sigma$ is an irreducible square-integrable  representation of $G$, then it is an essential component all by itself, labelled by the associate class $[G,\sigma]$. The irreducible spin module matched to $[G,\sigma]$ is   the unique one, up to not-necessarily-grading-preserving isomorphism, for which 
\[
[H\otimes S]^K \ne 0.
\]
Moreover if  $\mu$ is  the highest weight of the irreducible and genuine representation $\module (S)$ of $\widetilde K$, then $\mu{+}\rho_K$   is the so-called \emph{Harish-Chandra parameter} of $\sigma$. Compare \cite[Thm. 9.3]{AtiyahSchmid} or  \cite[\S2]{LafforgueICM}.
\end{example}

\begin{example}
The reduced C*-algebra for $G{=}{SL}(2,\R)$  was mostly described in \cite[Ex. 6.10]{CCH1} (and in several earlier works). The decomposition of the intertwining groups $W_\sigma$ as semidirect products $W'_\sigma\rtimes R_\sigma$ was not discussed there, but one may determine by direct computation  that the only inessential component in the tempered dual  is the spherical principal series component. The essential associate classes are therefore of two types: the discrete series and the odd principal series.

If $[P,\sigma]$ is a discrete series component with Harish-Chandra parameter $n\in \Z$, $n{\ne }0$, then the matching spin module $S_n$  is  
\[
S_n=S_{\text{irr}}\otimes\C_{n},
\] 
where $\C_{n}$ denotes  the weight-$n$  irreducible representation of $\operatorname{SO}(2)$, viewed as a genuine representation of $\widetilde K \cong K{\times} \Z_2$.  This may be computed directly, but the result is in line with Example~\ref{ex-discrete-series-and-h-c-parameters} above. 
If $[P,\sigma]$ is the odd principal series component of the tempered dual, then  the matching spin module
is $S_0=S_{\text{irr}}\otimes\C_{0}$; 
the matching conditions in Definition~\ref{def-matching} may again be checked by direct computation.   
\end{example}
 
\begin{example}
If $G$ is a complex reductive group, then all essential associate classes are attached to the minimal parabolic $P_\text{min}=MAN$, for which $M$ is a maximal torus in a maximal compact subgroup of $G$. The Connes-Kasparov isomorphism was established  in \cite{PeningtonPlymen}.  A bit more generally, the matching theorem was established for semi-simple Lie groups having only one conjugacy class of Cartan subgroups by Valette in \cite{Valette85} (see Theorem 3.12).  In the complex case the correspondence provided by the matching theorem is as follows: the spin module $S_{\text{irr}}\otimes V_\mu$, where $V_\mu$ is irreducible with highest weight $\mu$, is matched to the associate class $[P_\text{min},\sigma]$, where the differential of  $\sigma\in\hat{M}$ is  $\mu{+}\rho_K$. 
\end{example}

\begin{example}
\begin{figure}[ht]
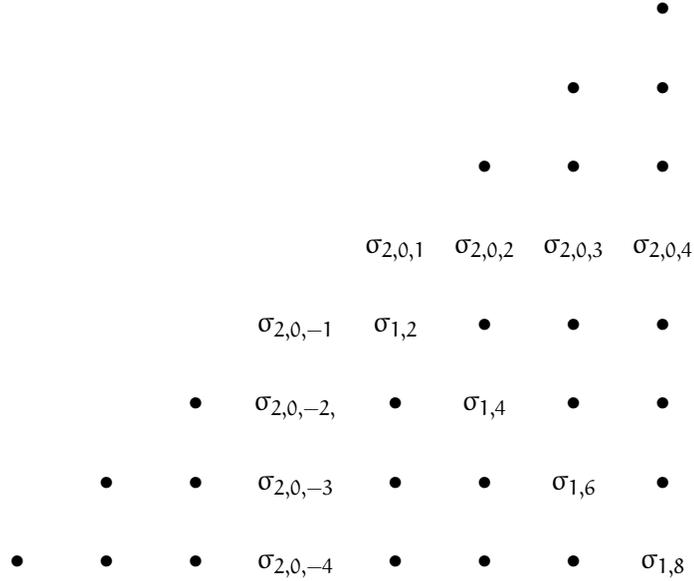

  \[ 
  \renewcommand{\arraystretch}{2.3}
\begin{array}{ccccccccc}
    & &&&&&&& \\
    & &&&&&&  &\Bullet \\
    & &&&&& & \Bullet& \Bullet \\
    & &&&&&\Bullet&\Bullet&\Bullet \\
    & &&&&\sigma_{2,0,1}&\sigma_{2,0,2}&\sigma_{2,0,3}&\sigma_{2,0,4} \\
    & &&&\sigma_{2,0,-1}&\sigma_{1,2}&\Bullet&\Bullet&\Bullet \\
    &         &         &\Bullet &\sigma_{2,0,-2,}&\Bullet &\sigma_{1,4}&\Bullet &\Bullet \\
    &         &\Bullet  &\Bullet &\sigma_{2,0,-3}&\Bullet &\Bullet&\sigma_{1,6}&\Bullet  \\
    &\Bullet  &\Bullet  &\Bullet &\sigma_{2,0,-4}&\Bullet &\Bullet & \Bullet &\sigma_{1,8} \\
\end{array} 
   \]
\caption{ 
The matching theorem for $Sp(4,\R)$.
The nodes in this diagram are the integer lattice points $(m,n)$ in the plane with $m{>}n$; the entry $\sigma_{2,2}$ appears in position  $(1,-1)$.
The $\sigma$-labels are some of the  discrete series attached to intermediate Levi subgroups of $Sp(4,\R)$, and together with the discrete series of $Sp(4,\R)$,  they give the full list of essential components in the tempered dual of $Sp(4,\R)$.  Each essential component is placed at the location $(m,n)=\mu{+}\rho_K$, where $\mu$ is the highest weight of the genuine irreducible representation of $\tilde K$   to which it is matched.   The  bullet points represent the discrete series  for $Sp(4,\R)$; their locations   are also  their Harish-Chandra parameters. 
}
\label{fig-sp4r-matching-theorem-summary}
\end{figure}
A more complicated and more interesting example is that of the real symplectic group  $G{=}Sp(4,\R)$.  

There are three components of the tempered dual  associated with the minimal parabolic subgroup $P_{\min} = MAN$ of $Sp(4,\R)$, given by three characters $\sigma_0$, $\sigma_1$ and $\sigma_2$ of the finite group $M$ (there are four characters altogether, but two lead to the same associate class).  None of the principal series components is essential.

There are two other associate classes of proper parabolic subgroups,  with Levi factors \[
L_1\cong GL(2,\R)\quad \text{and} \quad 
L_2 \cong GL(1,\R){\times} SL(2,\R).
\]
The compactly generated subgroups $M_1\subseteq L_1$ and $M_2\subseteq L_2$  both carry discrete series.
The group $M_1$ is isomorphic to $SL_{\pm}(2,\R)$, and its  discrete series are parametrized by positive integers; let us write these representations as $\sigma_{1,k}$ ($k>0$).  
The  group $M_2$ is isomorphic to $ O(1){\times} SL(2,\R)$,  and its  discrete series are para\-metrized by pairs $(\ell,k)$  where $\ell\in \Z_2$ and $k\in \Z$, $k{\ne } 0$; let us write these representations  as $\sigma_{2,\ell,k}$.  
Altogether, the discrete series 
\[
 \sigma_{1,k}, \quad (k\in \Z,\,\, k>0)\quad \text{and} \quad  \sigma_{2,\ell,k}\quad (\ell\in \Z_2,\,\, k\in\Z,\,\, k\ne 0)
 \]
label the  components in the tempered dual of $G$ that are associated to the intermediate parabolic subgroups, and the above representations label these components without repetition. The essential components associated to the intermediate parabolic subgroups are 
\[
 \sigma_{1,k}, \quad (k\in 2\Z,\,\, k>0)\quad \text{and} \quad  \sigma_{2,0,k}\quad (k\in\Z,\,\, k\ne 0).
 \]
 
Finally, there are the discrete series of $Sp(4,\R)$.
The maximal compact subgroup of $Sp(4,\R)$ is $K{\cong} U(2)$.  Using the diagonal maximal torus in $U(2)$, the irreducible representations of $K$ may be identified, via highest weights, with pairs of integers $(m,n)$ such that $m{\ge} n$, and the Harish-Chandra parameters of the discrete series may be identified with pairs $(m,n)$ with 
\[
m> n, \quad m\ne 0, \quad n\ne 0\quad \text{and} \quad m\ne -n.
\]
To describe the matching theorem, it is convenient to associate to the spin module 
$S_{\text{irr}}\otimes V_\mu$  (where $V_\mu$ is irreducible with highest weight $\mu$)  the parameter $\mu {+}\rho _K$. Compare Example~\ref{ex-discrete-series-and-h-c-parameters}. The parameters $\mu{+}\rho_K$  range over all   pairs  of integers $(m,n)$ with  $m{>}n$, and using  the $\mu{+}\rho_K$  parametrization, the matching theorem is illustrated in Figure~\ref{fig-sp4r-matching-theorem-summary}.  

The computations involved in  checking the matching theorem   are greatly simplified by David Vogan's theory of minimal $K$-types, and we shall say more about this in the second paper of the series \cite[Thm. 8.4]{MatchingTheorem}.
\end{example}

\section{First Proof of the Connes-Kasparov Isomorphism}
\label{sec-first-proof}

In this section we shall use the Matching Theorem formulated in the previous section to prove that the Connes-Kasparov index homomorphism is an isomorphism.  We shall follow the shortest route to do so, which uses the fact, proved by Kasparov, that the index homomorphism is a split-injective homomorphism of abelian groups (that is,   the index homomorphism has a left-inverse). While this is certainly a significant new ingredient in the proof,  injectivity is a considerably simpler  and more accessible result than surjectivity.  (In any case, in the next section we shall take a different approach to the proof of the Connes-Kasparov isomorphism that avoids Kasparov's result.)

Kasparov proved split injectivity  in a much broader context than the one we are considering here---involving both continuous and discrete groups---in the course of proving groundbreaking results on the Novikov conjecture in differential topology. But let us record his result as it applies in our case:

\begin{theorem*}[\cite{KK_Novikov}] 
The Connes-Kasparov index morphism
\[
R_{\text{\rm spin}}(K,\mathfrak{s}) \longrightarrow K_{\dim(G/K)}(C^*_r (G))
\]
is a split injection of abelian groups.
\end{theorem*}

\subsection*{Proof of the  Connes-Kasparov Isomorphism Theorem Using Split Injectivity}

To begin with,  Theorem~\ref{thm-k-theory-c-star-g-summarized} shows in particular that 
\[
K_{\dim(G/K)+1} \bigl (C^*_r (G)\bigr ) =0,
\]
which is one of the assertions in Theorem~\ref{thm-connes-kasparov-isomorphism}. The main task is to show that the index homomorphism 
\[
  R_{\mathrm{spin}}(K,\mathfrak{s}) \longrightarrow K_{\dim(G/K)}(\Csr  (G))
\]
is an isomorphism of abelian groups.

The $C^*$-algebra isomorphism  in Theorem~\ref{thm-CsrG} determines a $K$-theory direct sum decomposition
\begin{equation}
    \label{eq-k-theory-direct-sum-decomposition}
K_{\dim (G/K)} \big ( C^*_r (G)\bigr ) 
\cong  \bigoplus_{[P,\sigma]}
K_{\dim (G/K)} \bigl ( C_0 (\mathfrak{a}^*_P, \mathfrak{K} (\Ind_P^G H_\sigma) ) ^{W_\sigma}\bigr )
\end{equation}
If $S$ is a spin-module for $(K,\mathfrak{s})$, then we shall denote by 
\[
\Index_{[P,\sigma]} \bigl (\slashed{D}_S\bigr ) \in K_{\dim (G/K)} \bigl ( C_0 (\mathfrak{a}^*_P, \mathfrak{K} (\Ind_P^G H_\sigma) ) ^{W_\sigma}\bigr )
\]
the $[P,\sigma]$-component in \eqref{eq-k-theory-direct-sum-decomposition} of the image of $S$ under the Connes-Kasparov index homomorphism.

\begin{lemma}
\label{lem-vanishing-index-when-unmatched}
Let $S$ be an irreducible spin module for $(K,\mathfrak{s})$.
If $[P,\sigma]$ and $S$ are unmatched, then $\Index_{[P,\sigma]} \big ( \slashed{D}_S) =0$.
\end{lemma}

\begin{proof}
If $[\Ind_P^G H_\sigma \otimes S ] ^K=0$, then certainly $\Index_{[P,
\sigma]}(\slashed{D}_S) =0$.  If $[\Ind_P^G H_\sigma \otimes S ] ^K$ is non-zero but $[P,\sigma]$ and $S$ are unmatched, then the operator $\slashed{D}_{\sigma,0,S}$ is nonzero on $[\Ind_P^G H_\sigma \otimes S ] ^K$.  Since the Dirac operator is self-adjoint, the square is also nonzero on $[\Ind_P^G H_\sigma \otimes S ] ^K$, and therefore, by Theorem~\ref{thm-square-of-principal-series-dirac-ops}, 
\[
\|S\|^2 - \| \infch (\sigma)\| ^2 >0
\]
But Theorem~\ref{thm-square-of-principal-series-dirac-ops} asserts more generally that 
\[
\slashed {D}_{\sigma, \varphi, S}^2 = \|S\|^2 - \| \infch (\sigma)\| ^2 + \|\varphi \|^2,
\]
and therefore $\slashed{D}_S^2$ is uniformly bounded below over the $[P,\sigma]$-component of $C^*_r(G)$.  The bounded operator $\slashed{F}_S$ is therefore invertible there,  and hence the index is zero.
\end{proof}

\begin{proof}[Proof of Theorem~\ref{thm-connes-kasparov-isomorphism}]
We shall use the Matching Theorem. Let $S$ be an irreducible spin module.  Lemma~\ref{lem-vanishing-index-when-unmatched} implies that the image of $S$ under the index homomorphism is concentrated in the summand in \eqref{eq-k-theory-direct-sum-decomposition} associated to the unique $[P,\sigma]$ to which $S$ is matched.  Since the index homomorphism is injective, the image there must be nonzero.  In fact, because the index homomorphism is split injective, while the summand is isomorphic to $\Z$, the image must be a generator. That is, the index homomorphism maps the basis of $R_{\spin}(K,\mathfrak{s})$ determined by the irreducible spin modules  to the basis determined up to signs by Theorem~\ref{thm-k-theory-0f-essential-component}. 
\end{proof}

\section{Second Proof of the Connes-Kasparov Isomorphism}
\label{sec-second-proof}
In this final section we shall study the Dirac operator $\slashed{D}_S$ in more detail, and by doing so give  a proof of the Connes-Kasparov isomorphism that is independent of Kasparov's split-injectivity result. This  is probably more in line with the approach that Wassermann intended to take, as sketched  in the note \cite{NoteWassermann}.

\subsection*{K-Theoretic Preliminaries}

As we have  seen, the Connes-Kasparov index homomorphism  carries the natural basis for $R_{\spin} (K,\mathfrak{s})$ to the natural basis\footnote{To be accurate, both bases are defined up to choices of signs.}  for the $K$-theory of $C^*_r (G)$ (labeled by the essential components of the tempered dual; see Theorem \ref{thm-k-theory-c-star-g-summarized}).  A striking feature of the Connes-Kasparov index  is that in fact it carries   natural basis elements to natural basis elements \emph{at the level of cycles}, and not merely at the level of $K$-theory classes.  In this section we shall  describe those cycles.

\begin{definition} 
Let $V$ be a finite-dimensional Euclidean vector space of dimension $d$. A \emph{Bott element} for $V$ consists of a finite-dimensional $\Z/2$-graded Hilbert space $S$ with
\[
\dim (S) = 
\begin{cases} 
2^{d/2} & \text{$d$ even} \\
2^{(d{-}1)/2} & \text{$d$ odd}
\end{cases}
\]
and an $\R$-linear map $v\mapsto D_v$ from $V$ into the  odd-graded, self-adjoint operators 
on $S$ 
such that $D^2_v =   \| v\|^2$ for  all $v \in V$.   When the dimension of $V$ is  odd, we also require that $S$ be equipped with a  symmetry $\gamma$ as in 
 \eqref{eq-odd-grading-symmetry} that anti-commutes with all $D_v$.
\end{definition}

It follows from the elementary theory of Clifford algebras that Bott elements are unique up to isomorphism. Each Bott element may be regarded as   a  Fredholm operator on the Hilbert $C_0(V)$-module $C_0(V, S)$ of the sort considered by Kasparov ($D$ is unbounded, but one can take the bounded transform to obtain a bounded Fredholm operator $F$ if preferred). There is therefore an index 
\[
\Index (D) \in K_{d} (C_0(V)).
\]
Here is one form of the Bott periodicity theorem (see \cite[Theorem 7 on p.547]{Kasparov_Kfunctor}): 

\begin{theorem}\label{thm-Bott element}
Let $V$ be a finite-dimensional Euclidean vector space of dimension $d$.  The $K_d$-group of $C_0(V)$ is freely generated by   the index of any Bott element, and the $K_{d+1}$-group is zero. 
\end{theorem}

\subsection*{Representation-Theoretic Preliminaries}
Now let  $[P,\sigma]$ be an essential associate class. 
As noted earlier, there 
  is a decomposition of the parabolically induced representation $\pi_{\sigma,0}$ into finitely many  irreducible subrepresentations, 
\begin{equation}
    \label{eq-decomposition-of-standard-into-irreducibles}
\Ind_P^G H_\sigma = \bigoplus_{\mu} X_\mu ,
\end{equation}
and the   index set in the direct sum is the set $\widehat R_\sigma$ of characters of the finite abelian group $R_\sigma$.  But we can index the sum in a different way using Vogan's theory of minimal $K$-types \cite{Voganbook}, and it  will be very useful to do so in what follows.  

It will not be important to present the precise definition of minimal $K$-type here.  It will suffice to recall that the \emph{$K$-types} of a representation $\pi$ of $G$ are the irreducible representations of $K$ that occur upon restriction of $\pi$ from $G$ to $K$, and that every representation has a finite number of \emph{minimal} $K$-types among these, which depend only on the set of all $K$-types in $\pi$.  

The deeper properties of minimal $K$-types that we shall use below are as follows:

\begin{theorem}
\label{thm-minimal-k-types}
Let $[P,\sigma]$ be an essential associate class, and let $S$ be the irreducible spin module to which it is matched.   
\begin{enumerate}[\rm (i)]

\item 
Each minimal $K$-type of $\Ind_P^G H_\sigma$ has multiplicity one\footnote{That is, the underlying irreducible representation of $K$ occurs precisely once in any decomposition of $\Ind_P^G H_\sigma$ into irreducible representations of $K$.}, and each irreducible direct summand $X_\mu$ of $\Ind_P^G H_\sigma$, as in \eqref{eq-decomposition-of-standard-into-irreducibles}, includes precisely one of these minimal $K$-types.

\item  If $X_\mu$ is any irreducible summand of $\Ind_P^G H_\sigma$,   then
\[
\dim  \bigl [ X_\mu \otimes S\bigr ]^K = 2^{ [(\dim(\mathfrak{a}_{\max}){+}1)/2 ] },
\]
where the brackets $[\,\,]$ in the exponent denote the integer part.

\item If $X_\mu$ is any irreducible summand of $\Ind_P^G H_\sigma$, and if $V_\mu\subseteq X_\mu $ is its minimal $K$-type, then the inclusion 
\[
[ V_\mu \otimes S ] ^K \longrightarrow [X_\mu \otimes S]^K
\]
is a vector space isomorphism.

\end{enumerate}
\end{theorem}

We shall prove this theorem in \cite[Sec.~8]{MatchingTheorem} (mostly by collecting results from elsewhere in the representation theory literature).

It follows from parts (i) and (iii) of the theorem, together with the direct sum decomposition \eqref{eq-decomposition-of-standard-into-irreducibles}, that if   $\{ V_\mu\}$ is the set of minimal $K$-types in $\Ind_P^G H_\sigma$, then the inclusion
\begin{equation}
    \label{eq-direct-sum-decomposition-of-dirac-cohomology}
 \bigoplus _{\mu} [V_\mu \otimes S]^K \longrightarrow 
 [\Ind_P^G H_\sigma \otimes S]^K 
 \end{equation}
 is a vector space isomorphism.  This gives a very concrete and convenient description of the space 
$[\Ind_P^G H_\sigma \otimes S]^K$.  The following lemmas examine the Dirac operators that act on this space.

 \begin{lemma}
 Let $[P,\sigma]$ be an essential associate class and let $S$ be the irreducible spin module to which it is matched. 
 The operators 
 \[
\slashed{D}_{\sigma,\varphi,S }\colon [V_\mu \otimes S]^K
\longrightarrow [V_\mu \otimes S]^K
\]
are linear functions of $\varphi \in\mathfrak{a}^{*}_P$.
\end{lemma}

\begin{proof}
The action of $\mathfrak{g}$ on the smooth vectors in any principal series representation space  such as $\Ind_P^G H_\sigma {\otimes} \C_{i \varphi}$ is   affine-linear in $\varphi$ (compare \cite[Prop.~11.47]{KnappVoganBook}), and so $\slashed{D}_{\sigma, \varphi, S}$ is affine-linear in $\varphi$. But since $S$ is matched to $(P,\sigma)$, the operator $\slashed{D}_{\sigma, 0, S}$ is zero.  So $\slashed{D}_{\sigma, \varphi, S}$ is actually linear in $\varphi$.
\end{proof}

  \begin{lemma}
 \label{lem-dirac-cohomology-computation}
 Let $[P,\sigma]$ be an essential associate class and let $S$ be the irreducible spin module to which it is matched.  Denote  by $\mathfrak{a}^{*,R_\sigma}_P\subseteq \mathfrak{a}^*_P$ the subspace that is fixed under the action of the group $R_\sigma$.
 The image of each  direct summand in \eqref{eq-direct-sum-decomposition-of-dirac-cohomology} is invariant under  the Dirac operators \[
 \slashed{D}_{\sigma,\varphi,S }\colon [\Ind_P^G H_\sigma \otimes S]^K  \to [\Ind_P^G H_\sigma \otimes S]^K 
 \]
 for all $\varphi \in \mathfrak{a}^{*,R_\sigma}_P$.
 \end{lemma}
 
 \begin{proof}
The representations $X_{\mu,\varphi}$ that appear in the direct sum decomposition
\begin{equation*}
\Ind_P^G H_\sigma{\otimes}\C_{i\varphi}  = \bigoplus _\mu X_{\mu, \varphi},
\end{equation*}
compare \eqref{eq-decomposition-of-ind-sigma-tensor-exp-i-nu}, 
 have the same $K$-isotypic decompositions as the representations $X_\mu$. Therefore    for every $\varphi \in \mathfrak{a}^{*,R_\sigma}_P$ the $K$-type $V_\mu$ appears in $X_{\mu,\varphi}$ as a minimal $K$-type, and the inclusion 
\[
[V_\mu \otimes S ] ^K \longrightarrow 
[X_{\mu,\varphi}\otimes S]^K
\]
is a vector space isomorphism, since $ [X_{\mu,\varphi}\otimes S]^K$  depends only on the $K$-structure of $X_{\mu,\varphi}$, and not on $\varphi$.    The Dirac operator $D_{\sigma,\varphi,S}$ certainly maps $[X_{\mu,\varphi}\otimes S]^K$ to itself, and so it maps $[V_\mu \otimes S ] ^K$ to itself, as claimed.
\end{proof}

\begin{theorem}\label{thm-index-is-bott-element}
If $[P,\sigma]$ is any essential associate class, and if   $S$ is the  irreducible spin module  to which it is matched,  then for any $\mu$ the family of Dirac operators 
\[
\slashed{D}_{\sigma,\varphi, S} \colon \bigl [  X_\mu  \otimes S\bigr ]^K 
\longrightarrow \bigl [   X_\mu  \otimes S\bigr ]^K \qquad (\varphi \in \mathfrak{a}_P^{*,R_\sigma}) 
\]
is a Bott element for $\mathfrak{a}_P^{*,R_\sigma}$.
\end{theorem}
\begin{proof} 
This follows from the preceding two lemmas and Theorem~\ref{thm-square-of-principal-series-dirac-ops}.
\end{proof}

\subsection*{Completion of the Second Proof of the Connes-Kasparov Isomorphism}
The Matching Theorem and Lemma~\ref{lem-vanishing-index-when-unmatched}  show that the Connes-Kasparov index morphism maps an irreducible spin module for $(K, \mathfrak{s})$ to the index of the family of Dirac operators
\begin{equation}
\label{eq-index-cycle-from-matching-theorem}
\slashed{D}_{\sigma, \varphi, S}: \bigl [  \Ind_P^G H_\sigma  \otimes S\bigr ]^K 
\longrightarrow \bigl [   \Ind_P^G H_\sigma   \otimes S\bigr ]^K,
\end{equation}
in $K_{\rm{dim}(G/K)}(C_0  (\mathfrak{a}^*_P, \mathfrak{K} (  \Ind_P^G H_\sigma)  ) ^{W_\sigma} )$, where $[P, \sigma]$ is the unique  essential associate class matched to $S$. To prove the Connes-Kasparov isomorphism, it remains to show that the homotopy class of the family \eqref{eq-index-cycle-from-matching-theorem} is a generator for the $K$-theory group.

By Theorem \ref{thm-index-is-bott-element}, for any summand $X_\mu$ of $\Ind_P^G H_\sigma$ the family of Dirac operators 
\[
\slashed{D}_{\sigma, \varphi, S}: \bigl [X_\mu  \otimes S\bigr ]^K 
\longrightarrow \bigl [X_\mu  \otimes S\bigr ]^K \qquad (\varphi \in \mathfrak{a}_P^{*,R_\sigma})
\]
is a Bott element for $\mathfrak{a}_P^{*,R_\sigma}$ and therefore its index is a 
generator of the $K$-theory group $K_{\rm{dim}(G/K)}(C_0(\mathfrak{a}_P^{*, R_\sigma}))$.  But \eqref{eq-index-cycle-from-matching-theorem}  is precisely the image of the cycle that defines the Connes-Kasparov index 
\[
\Index_{[P,\sigma]} (\slashed{D}_S) 
\in K_{\rm{dim} (G/K)}
\bigl (C_0 \bigl (\mathfrak{a}^*_P, \mathfrak{K} (  \Ind_P^G H_\sigma) \bigr ) ^{W_\sigma} \bigr)
\]
under the $K$-theory isomorphism in Theorem~\ref{thm-k-theory-0f-essential-component2}.  The proof is complete.

\subsection*{Acknowledgements}
This research was supported by NSF grants DMS-1952669 (NH), DMS-1800667,  DMS-1952557 (YS), DMS-1800666 and DMS-1952551 (XT).
Part of the research   was carried out within the online Research Community on Representation Theory and Noncommutative Geometry sponsored by the American Institute of Mathematics.

\bibliographystyle{alpha}
\bibliography{biblio}

\end{document}